\documentclass[12pt, fleqn, reqno]{amsart}%

\usepackage[top=1in, bottom=1in, left=1in, right=1in]{geometry}

\usepackage{mathtools,mathrsfs,amsthm,amsfonts,amssymb,graphicx,appendix,cancel}
\usepackage{xspace,enumerate,comment,bbm,nth}

\usepackage{tikz,pgfplots}

\usepackage{hyperref}
\usepackage[normalem]{ulem}

\usepackage[active]{srcltx}

\usepackage{tikz,pgfplots}

\usepackage[normalem]{ulem}

\usepackage[active]{srcltx}

\newenvironment{claim}{{\sc Claim:}}

\newtheorem{theorem}{\sc Theorem}

\newtheorem{corollary}[theorem]{\sc Corollary}
\newtheorem{definition}[theorem]{\sc Definition}
\newtheorem{lemma}[theorem]{\sc Lemma}
\newtheorem{proposition}[theorem]{\sc Proposition}

\theoremstyle{plain}
\newtheorem{prop}[theorem]{\sc Proposition}

\theoremstyle{remark}
\newtheorem{remark}[theorem]{\sc Remark}

\newcommand{\A}{\mathcal A}

\newcommand{\K}{\mathcal K}
\newcommand{\M}{\mathcal M}
\newcommand{\Mtilde}{\widetilde{\mathcal M}}
\newcommand{\Mis}{\widetilde{\mathfrak{M}}}
\newcommand{\R}{\mathbb{R}}

\newcommand{\N}{\mathbb{N}}

\newcommand{\PP}{\mathscr{P}}

\newcommand{\parts}{\mathscr{P}}

\newcommand{\T}{\mathbb{T}}

\newcommand{\comp}{\mbox{\scriptsize  $\circ$}}
\newcommand{\D}[1]{\mbox{\rm #1}}
\newcommand{\eps}{\varepsilon}
\renewcommand{\epsilon}{\varepsilon}

\renewcommand{\emptyset}{\varnothing}

\newcommand{\weakst}{\stackrel{\ast}{\rightharpoonup}}

\begin{document}
\baselineskip=15pt

\title[Convergence/divergence of discounted solutions]{Convergence/divergence phenomena in the vanishing discount limit of Hamilton-Jacobi equations
}\author[A.\ Davini, P.\ Ni, J.\ Yan\and M.\ Zavidovique]{Andrea Davini, Panrui Ni, Jun Yan\and Maxime Zavidovique}

\address{Andrea Davini\\ Dipartimento di Matematica\\ {Sapienza} Universit\`a di
  Roma\\ P.le Aldo Moro 2, 00185 Roma\\ Italy}
\email{davini@mat.uniroma1.it}

\address{Panrui Ni\\ Sorbonne Universit\'e, Universit\'e de Paris Cit\'e, CNRS, Institut de Math\'ematiques de Jussieu-Paris Rive Gauche \\Paris 75005\\ France }
\email{panruini@imj-prg.fr}

\address{Jun Yan\\ School of Mathematical Sciences\\ Fudan University\\ Shanghai 200433\\ China}
\email{yanjun@fudan.edu.cn}

\address{Maxime Zavidovique\\ Sorbonne Universit\'e, Universit\'e de Paris Cit\'e, CNRS, Institut de Math\'ematiques de Jussieu-Paris Rive Gauche \\Paris 75005\\ France}
\email{mzavidovi@imj-prg.fr}

\subjclass{ }

\vspace{-6ex}
\begin{abstract}
We study the asymptotic behavior of solutions of an equation of the form
\begin{equation}\label{abs}\tag{*}
G\big(x, D_x u,\lambda u(x)\big) = c_0\qquad\hbox{in $M$}
\end{equation}
on a closed Riemannian manifold $M$, where $G\in C(T^*M\times\R)$ is convex and superlinear in the gradient variable, is globally Lipschitz but not monotone in the last argument, and $c_0$ is the critical constant associated with the Hamiltonian $H:=G(\cdot,\cdot,0)$.
By assuming that $\partial_u G(\cdot,\cdot,0)$ satisfies a positivity condition of integral type on the Mather set of $H$, we prove that any equi-bounded family of solutions of \eqref{abs} uniformly converges to a distinguished critical solution $u_0$ as $\lambda \to 0^+$.
We furthermore show that any other possible family of solutions uniformly diverges to $+\infty$ or $-\infty$. We then look into {the linear} case $G(x,p,u):=a(x)u + H(x,p)$ and prove that the family $(u_\lambda)_{\lambda \in (0,\lambda_0)}$ of maximal solutions to \eqref{abs} is well defined and equi-bounded for $\lambda_0>0$ small enough. When $a$ changes sign and enjoys a stronger localized positivity assumption, we show that equation \eqref{abs}
does admit other solutions too, and that they all  uniformly diverge to $-\infty$ as $\lambda \to 0^+$.
This is the first time that converging and diverging families of solutions are shown to coexist in such a generality.
\end{abstract}

\maketitle

\section*{Introduction}
In this paper we are concerned with the asymptotic behavior of solutions of an equation of the form
\begin{equation}\label{intro E}\tag{E$_\lambda$}
G\big(x, D_x u,\lambda u(x)\big) = c_0\qquad\hbox{in $M$}
\end{equation}
{posed} on a closed Riemannian manifold $M$, where $G\in C(T^*M\times\R)$ is convex and superlinear in the gradient variable, is globally Lipschitz in the last argument, and $c_0$ is the critical constant associated with the Hamiltonian $H:=G(\cdot,\cdot,0)$. We refer the reader to Section \ref{sec weak kam} for the definition of $c_0$ and of the other related objects coming from weak KAM Theory that will be mentioned in this introduction.
The monotonicity condition on $G$ in the last argument, that is standard for these kind of equations, is dropped here in favor of the following much weaker integral condition
\begin{itemize}
\item[\textbf{(L5)}] \qquad $\displaystyle \int_{TM} \dfrac{\partial L_G}{\partial u}(x,v,0)\,d\tilde{\mu}(x,v)<0
\qquad
\hbox{for all $\tilde\mu\in\Mis$,}$
\end{itemize}
where $L_G$ is the convex conjugate function of $G$, $\Mis$ denotes the set of Mather measures for $L_G(\cdot,\cdot,0)$, and $G$ (and hence $L_G$) satisfies a $C^1$-type regularity condition near $u=0$, see conditions (G4) and (L4) in Section \ref{sec general}.
Under these assumptions, we prove that  any equi-bounded family of solutions of \eqref{intro E} uniformly converges to a distinguished critical solution $u_0$ as $\lambda \to 0^+$. We furthermore show that any family consisting of other possible solutions uniformly diverges to $+\infty$ or $-\infty$.

We underline that the conditions presented above on $G$ are not sufficient to guarantee the existence and uniqueness of viscosity solutions to \eqref{E}.
This is due to the fact that, without global strict monotonicity of $G$ with respect to $u$, there is no comparison principle. This also makes unenforceable Perron's method, which is the technique customarily employed to prove existence of solutions.

The general issue of existence and uniqueness of such solutions is subsequently addressed in the paper in the linear case $G(x,p,u):=a(x)u + H(x,p)$ under the minimal hypotheses on $a\in C(M)$ and $H\in C(T^*M)$
that guarantee that the conditions on $G$ and $L_G$ mentioned above are in force. We prove that the family $(u_\lambda)_{\lambda \in (0,\lambda_0)}$ of maximal solutions to \eqref{intro E} is well defined and equi-bounded for $\lambda_0>0$ small enough. When $a$ changes sign and $a\geqslant 0$ in a neighborhood of the projected Aubry set $\A$, we show that equation \eqref{intro E} does admit other solutions too  that uniformly diverge to $-\infty$ as $\lambda \to 0^+$. If we additionally assume $a>0$ on $\A$, we furthermore show that any family made up of solutions to \eqref{intro E} that differ from the maximal ones uniformly diverge to $-\infty$. Incidentally, this completely solves the vanishing discount problem for this model case under the sole assumption that $a>0$ on $\A$  in view of the results established in \cite{Z},  where $a$ was additionally assumed nonnegative on $M$.

Condition (L5) was introduced in \cite{CFZZ} and therein employed to solve the vanishing discount problem for an equation of the form
\eqref{intro E} under the same set of assumptions considered herein, plus the additional requirement that $G$ is globally non-decreasing in $u$.
Condition (L5) can be read as a strict monotonicity condition on $G$ with respect to $u$, and this is transparent in the linear case
$G(x,p,u):=a(x)u+H(x,p)$. What we find striking about the output of our study is the fact that (L5) is a very weak requirement: it implies that $a$ has to be strictly positive only on some portions of the projected Mather set $\M$, where the latter is the minimal closed set that contains the projection of the supports of all Mather measures.
This set $\M$ can be very small, such as a finite set of points, see Remark \ref{oss small A}. Furthermore, it has been conjectured  by Ma\~ n\' e \cite{Mane} that for generic Hamiltonians $H$
both $\mathcal A$ and $\mathcal M$ coincide with the support of a closed curve.  Many results have been obtained in this direction,  see for example \cite{CFR} and the references therein, showing that condition (L5) generically leaves a lot of space for $a$ to take negative values.

The main results proved in this paper keep holding when the superlinearity condition on $G$ is relaxed in favor of a simple coercivity. We have decided not to pursue this generalization here since that would add additional technicalities with the drawback of hiding the ideas at the base of our work, see Remark \ref{oss coercive} for further details.

\subsection*{History of the problem}
The so-called {\em ergodic approximation} is a technique introduced in \cite{hom} to study the existence of solutions of the Hamilton-Jacobi equation\footnote{All solutions in the paper are meant in the viscosity sense. The definition will be provided later.}
\begin{equation}\label{hjc}
H(x,D_x u)=c\qquad \hbox{in $M$}
\end{equation}
on the flat $d$-dimensional torus $M:=\T^d\simeq\mathbb R^d/\mathbb Z^d$, where the Hamiltonian $H$ is a continuous function on $T^*M$, coercive in the gradient variable, uniformly with respect to $x\in M$,  and $c$ is a real number. Let $\lambda>0$ and $u_\lambda$ be the unique solution of
\[\lambda u(x)+H(x,D_x u)=0 \qquad\hbox{in $M$}.
\]
According to \cite{hom}, {the functions $-\lambda u_{\lambda}$ uniformly converge on $M$,  as $\lambda\to 0^+$, to a constant $c_0$. Furthermore, the solutions $(u_\lambda)_{\lambda>0}$ are equi-Lipschitz, yielding, by the Arzel\'a-Ascoli Theorem, that the functions
$u_{\lambda}-\min_{x\in M}u_{\lambda}$ uniformly converge, along subsequences as $\lambda\to 0^+$, to a
solution of (\ref{hjc}) with $c$ equaling $c_0$.}  The constant $c_0$ is called {\em critical value} of $H$ and is characterized by the property of being the unique constant $c\in\mathbb R$ such that (\ref{hjc}) admits solutions.
%
At that time, it was not clear if different converging sequences yield the same limit. Some constraints on the possible limit solutions were subsequently found in \cite{G,IM}, but the breakthrough came with the work
\cite{Da4}, where the authors proved that the unique solution $u_\lambda$ of
\begin{equation}\label{lde}\tag{${\textrm{HJ}}_\lambda$}
\lambda u(x)+H(x,D_x u)=c_0\qquad\hbox{in $M$}
\end{equation}
converges to a distinguished solution of
\begin{equation}\label{e0}\tag{HJ$_0$}
  H(x,D_x u)=c_0\qquad\hbox{in $M$},
\end{equation}
as $\lambda\rightarrow 0^+$ under the sole additional assumption that $H$ is convex in the gradient variable. The proof relies on techniques and tools issued from weak KAM Theory, in particular on the concept of Mather measure, and it works whenever $M$ is a closed Riemannian manifold. This kind of problem is also known as the {\em vanishing discount problem}. When the convexity condition on $H$ is dropped,  the functions $u_\lambda$ may not converge, as it was pointed out in \cite{Z2} through a counterexample posed on the 1-dimensional torus.

As a nonlinear generalization  (see \cite{V1,V3,GMT} and \cite{V2}), one can study the uniform convergence of the unique solution of
\[H_\lambda\big(x,D_x u,u(x)\big)=c_0\qquad\hbox{in $M$},\]
as $\lambda\rightarrow 0^+$, where $H_\lambda(x,p,u)$ is strictly increasing in $u$, and uniformly converges to $H(x,p)$ on compact sets as $\lambda\rightarrow 0^+$. This kind of problem is called the {\em vanishing contact structure problem}. The vanishing discount problem falls in this framework as a particular case by choosing $H_\lambda (x,p,u):= \lambda u+H(x,p)$.

The asymptotic convergence result {has been subsequently established in} many different situations. For the second order case, one can refer to \cite{IMT,IMT2,MT,Zh}. For the discrete case, one can refer to \cite{Da1,n1,Zbook} and also \cite{AZa} in the context of twist maps. For the similar problem in the mean field game theory, one can refer to \cite{CP}. For the weakly coupled Hamilton-Jacobi systems, one can refer to \cite{Da5,Da6,Ish4,Ish5}. For the non-compact setting, one can refer to \cite{Da7,Ish6}.

A natural and challenging question is to weaken the hypothesis on the monotonicity of the Hamiltonian. A first degenerate case was studied in \cite{Z}, where the author considered the convergence of the solution of
\[\lambda a(x)u(x)+H(x,D_x u)=c_0\qquad\hbox{in $M$}\]
as $\lambda\rightarrow 0^+$, where $a(x)\geqslant  0$ on $M$, and $a(x)>0$ on the projected Aubry set of $H$. Inspired by the works \cite{V3, Z}, the authors studied  in \cite{CFZZ}  the vanishing discount problem for contact Hamilton-Jacobi equations of the form \eqref{intro E}, where the positivity hypothesis on $a$ assumed in \cite{Z} is weakened and generalized by introducing the non-degeneracy integral condition (L5).
This work highlights once more that the concept of Mather measure plays a central role in the convergence result.

It is worth pointing out that all works mentioned above required a global non-decreasing hypothesis of the Hamiltonians in $u$. If the discounted equation is not increasing in the unknown function $u$, solutions may even not exist, and, if they exist, they may be not unique. In \cite{Da3,V4}, the authors discussed the uniform convergence of the minimal solution of (\ref{lde}) as $\lambda\rightarrow 0^-$. For the non-monotone vanishing discount problem, the second author provided the first example in \cite{n3} of nonconvergence. In this example, there exist a convergent family of solutions and a divergent family of solutions at the same time. This phenomenon is new comparing with all the previous works in this direction. In the present paper, we show that the example in \cite{n3} is in fact a very general phenomenon when the Hamiltonian is continuous, convex and superlinear in the fibres.

Let us conclude by mentioning that the type of problems we study are also closely linked to optimization problems in economics. The discount factor then models the effect of time through interest rates or inflation. Negative interest rates or deflation have been studied by economists (see for instance \cite{Kocher}). Our results then give possible asymptotics in the presence of coexisting inflation and deflation.

\subsection*{Presentation of our results}
We present here our main results. Section \ref{sec general} contains our analysis on the asymptotic behavior of possible solutions of
a general contact Hamilton-Jacobi equation of the form \eqref{intro E} when the discount factor $\lambda$ goes to $0$. The Hamiltonian $G(x,p,u)$ is assumed convex and superlinear in $p$, and globally Lipschitz in $u$, see conditions (G1-3) in Section \ref{sec general}. It  is also assumed to satisfy a $C^1$-type regularity condition in $u$ near $u=0$, see condition (G4) in Section \ref{sec general}.
The latter is for instance satisfied when the map $u\mapsto G(x,p,u)$ is $C^1$ in a neighborhood of $u=0$ in the following sense:
\begin{itemize}
\item[\textbf{(G4$'$)}]  there exists $\varepsilon>0$ such that $\frac{\partial G}{\partial u}(x,p,u)$ exists for all $(x,p,u)\in T^*M\times (-\eps,\eps)$
and is continuous in $T^*M\times (-\eps,\eps)$.
\end{itemize}
Let us consider
\begin{equation}\label{E}\tag{E$_\lambda$}
  G\big(x,D_xu,\lambda u(x)\big)=c_0 \qquad\hbox{in $M$},
\end{equation}
and the limit equation
\begin{equation}\label{E0}\tag{${\textrm{E}}_0$}
  G(x,D_xu,0)=c_0 \qquad\hbox{in $M$},
\end{equation}
where $c_0$ is the critical value associated with $H:=G(\cdot,\cdot,0)$. We will furthermore assume the non-degeneracy integral condition (L5), where
$\Mis$ denotes the set of Mather measures for $L:=L_G(\cdot,\cdot,0)$.

The main results contained in Section \ref{sec general} can be summarized as follows, see Theorems \ref{thm1} and \ref{thm Maxime bis}.

\begin{theorem}\label{intro thm1}
Under the previous assumptions, there exist a viscosity solution $u_0$ of \eqref{E0} and functions $\varphi:(0,1)\to\mathbb (-\infty,+\infty]$,  $\psi:(0,1)\to\mathbb [-\infty,+\infty)$ and $\theta:(0,1)\to\mathbb R$ with
\[\lim_{\lambda\to 0}\psi(\lambda)=-\infty, \quad\quad \lim_{\lambda\to 0}\varphi(\lambda)=+\infty,\quad \quad \lim_{\lambda\to 0}\theta(\lambda)=0\]
such that, if $v_\lambda$ is a solution of (\ref{E}) for some $\lambda>0$, then either one of the following alternatives occurs:
\begin{itemize}
\item[\em (i)] $v_\lambda\leqslant  \psi(\lambda)$;\smallskip
\item[\em (ii)] $v_\lambda\geqslant  \varphi(\lambda)$;\smallskip
\item[\em (iii)] $\|v_\lambda-u_0\|_\infty\leqslant  \theta(\lambda)$.
\end{itemize}
\end{theorem}

We stress that all these three behaviors can happen at the same time for a properly chosen Hamiltonian $G$. Indeed, consider $G(x,p,u) := \sin( u) + \|p\|_x^2$. It is easlily seen that $G$ verifies all of the above conditions. The limit Hamiltonian is $H(x,p):=G(x,p,0)
= \|p\|_x^2$, its critical constant is $c_0 =0$ and the constant functions are the only solutions to $H(x,D_xu) = 0$ in $M$. Moreover, for all $\lambda>0$, the three constant functions $u_\lambda = 0$, and $u_\lambda ^\pm = \pm \frac{\pi}{\lambda}$ are solutions to the discounted equation.\smallskip

Theorem \ref{intro thm1} shows in great generality that the only possible asymptotic behavior of families of solutions $(v_\lambda)_{\lambda \in (0,\lambda_0)}$ is, up to subsequences, either to uniformly diverge to $\pm\infty$, or to uniformly converge to a specific solution $u_0$ of \eqref{E0}. We also provide two characterizations of $u_0$, see Theorems \ref{thm1} and \ref{thm2}.\smallskip

The general problem of existence  and uniqueness of such solutions is addressed in Section \ref{sec model case}. Here we consider the linear case $G(x,p,u) := a(x)u+H(x,p)$ under the minimal hypotheses on $a\in C(M)$ and $H\in C(T^*M)$
that guarantee that the conditions on $G$ and $L_G$ presented above are in force.
The discounted equation is then
 \begin{equation}\label{VH}\tag{${\textrm{HJ}}_\lambda$}
  \lambda a(x)u(x)+H(x,D_x u)=c_0\qquad\hbox{in $M$},
\end{equation}
with limit equation
\begin{equation}\label{e0}\tag{HJ$_0$}
  H(x,D_x u)=c_0 \qquad\hbox{in $M$}.
\end{equation}

We prove the following existence and convergence result:

\begin{theorem}
Under the previous hypotheses, there is $\lambda_0>0$ such that, for all $\lambda\in (0,\lambda_0)$, the
equation \eqref{VH} admits a maximal viscosity solution $u_\lambda\in C(M)$. Moreover, the family $(u_\lambda)_{\lambda \in (0,\lambda_0)}$ is equi-bounded, hence it uniformly converges to $u_0$ as $\lambda \to 0^+$.
 \end{theorem}

The previous theorem excludes the possibility of families of solutions that uniformly diverge to $+\infty$, but leaves open the possibility of  families uniformly diverging to $-\infty$. By strengthening the non-degeneracy integral condition (L5) with a pointwise positivity condition on $a$ on the projected Aubry set $\A$, we are able to improve the statement of Theorem \ref{intro thm1} as follows.

\begin{theorem}
Let us additionally assume that $a\geqslant 0$ in a neighborhood of $\mathcal A$ and that there exist $x_0\in M$ such that $a(x_0)<0$. Then there exists a family of solutions $(v_\lambda)_{\lambda\in (0,\hat \lambda)}$ to (\ref{VH}) for some $\hat\lambda\in (0,1)$ uniformly diverging to $-\infty$ as $\lambda \to 0^+$.
\end{theorem}

 When the previous hypothesis is reinforced, we obtain a stronger conclusion:

 \begin{theorem}
Let us additionally assume that $a>0$ on $\A$.  Then any family  $(v_\lambda)_{\lambda\in (0,\lambda')}$ of solutions to \eqref{VH} satisfying $v_\lambda \neq u_\lambda$ for all $\lambda\in (0,\lambda')$, with $\lambda'\in (0,1)$, uniformly diverges to $-\infty$.
 \end{theorem}
 Namely, in this last case, $(u_\lambda)_\lambda$ is the only converging family of solutions.\medskip

Let us conclude this presentation by stressing that, combining this analysis with the results of \cite{Z}, we fully understand the asymptotic behavior of solutions to \eqref{VH} when  $a>0$ on $\A$. The output is the following:
\begin{itemize}
\item[(a)] if $a\geqslant  0$ on the whole $M$, (this situation was considered in \cite{Z}) then, for all $\lambda>0$, there is a unique solution $u_\lambda$ to \eqref{VH}, and the family $(u_\lambda)_\lambda$ converges as $\lambda \to 0^+$.\smallskip
\item[(b)]  if there exist a point $x_0\in M$ such that $a(x_0)<0$, (this situation is discussed in the present paper), then we can find a $\lambda_0>0$ small enough such that equation \eqref{VH} admits at least two solutions for every $\lambda\in (0,\lambda_0)$. The family $(u_\lambda)_{\lambda\in (0,\lambda_0)}$ of maximal solutions uniformly converges to $u_0$ as $\lambda \to 0^+$. Any family of other solutions $(v_\lambda)_\lambda$ uniformly diverges to $-\infty$ as $\lambda \to 0^+$.
\end{itemize}

\numberwithin{theorem}{section}
\numberwithin{equation}{section}

\section{Preliminaries}\label{Sec2}

\subsection{Notation}
\begin{itemize}
\item Throughout this paper, we assume that $M$ is a closed, connected and smooth Riemannian manifold.
\item We fix $g$ an auxiliary Riemannian metric on $M$.  Let $d(x,y)$ be the distance between $x$ and $y$ in $M$ induced by $g$. By compactness of $M$ our results are independent on the choice of $g$.
\item Let diam$(M)$ be the diameter of $M$.
\item We denote by $TM$ and $T^*M$ the tangent and cotangent bundle over $M$ respectively. We denote by $(x,p)$ and $(x,v)$ points of $T^*M$ and $TM$ respectively.
\item Let $\pi:TM,T^*M\to M$ denote both canonical projections, the context will make it clear which one is considered.
\item We denote by $\|\cdot\|_x$ the norm on both $T_xM$ and $T_x^*M$ induced by $g$.
\item If $N$ is a smooth manifold, we will denote by $C(N)$ the Polish space of continuous functions from $N$ to $\R$ endowed with the metric
of local uniform convergence on $N$. We will denote by  $C_c(N)$ \big(resp. $C^1(N)$\big) the set of compactly supported (resp., $C^1$) functions from $N$ to $\R$.
\item We denote by $\parts (TM)$ the space of Borel probability measures on $TM$ endowed with the weak-$*$ topology coming from the dual $\big(C_c(TM)\big)'$.
\item We denote $C_\ell(TM)$ the set of continuous functions $g: TM\to \R$ with at most linear growth meaning that
\[
\sup\limits_{(x,v)\in TM}\frac{|g(x,v)|}{1+\|v\|_x}<+\infty.
\]
This last quantity defines a norm $\|g\|_\ell$ on the vector space $C_\ell(TM)$.
\item $\N$ denotes the set of positive integers.
\end{itemize}

\subsection{Viscosity solutions.}
We start by recalling the notion of viscosity solution.

\begin{definition}\rm
Let $G : T^*M\times \R \to \R$ be a continuous function, $c\in \R$, and consider the equation
\begin{equation}\label{eq def viscosity}
G\big(x,D_x u , u(x) \big) =c  \qquad\hbox{in $M$}.
\end{equation}
\begin{itemize}
\item[(a)] We say that $u\in C(M)$ is a {\em viscosity subsolution} of \eqref{eq def viscosity}, denoted by
\begin{equation*}
G\big(x,D_x u , u(x) \big) \leqslant   c  \qquad\hbox{in $M$},
\end{equation*}
if, for all $\varphi \in C^1(M)$ and $x_0\in M$ such that $u-\varphi$ has a local maximum at $x_0$, we have
$G\big(x_0,D_{x_0} \varphi  , u(x_0)\big ) \leqslant   c$. Such a function $\varphi$ is termed {\em supertangent to $u$ at $x_0$.}\smallskip
\item[(b)] We say that $u\in C(M)$ is a {\em viscosity supersolution} of \eqref{eq def viscosity}, denoted by
\begin{equation*}
G\big(x,D_x u , u(x) \big) \geqslant   c \qquad\hbox{in $M$},
\end{equation*}
if, for all $\varphi \in C^1(M)$ and $x_0\in M$ such that $u-\varphi$ has a local minimum at $x_0$, then
$G\big(x_0,D_{x_0} \varphi  , u(x_0) \big) \geqslant   c$.  Such a function $\varphi$ is termed {\em subtangent to $u$ at $x_0$.}\smallskip
\item[(c)] We say that $u\in C(M)$ is a {\em viscosity solution} of \eqref{eq def viscosity} if it is both a viscosity sub and supersolution.
\end{itemize}

\end{definition}

In this paper, solutions,  subsolutions, supersolutions will be always meant in the viscosity sense and implicitly assumed continuous. We recall that, if $u$ is $C^1$ on an open set $U$, then it is a viscosity solution (resp. subsolution, supersolution) in $U$ if and only if it is a pointwise solution (resp. subsolution, supersolution) in $U$.\smallskip

The following stability result is well known, see for instance \cite{barles}.

\begin{proposition}\label{prop stability}
Let $(G_n)_n$ and $(u_n)_n$ be two sequences of functions in $C(T^*M\times\R)$ and $C(M)$, respectively, such that $u_n$ is a
subsolution (resp. supersolution, solution) of \eqref{eq def viscosity} with $G:=G_n$, for each $n\in\N$. If $u_n\to u$ in $C(M)$  and $G_n\to G$ in $C(T^*M\times\R)$ as $n\to +\infty$, then $u$ is a subsolution (resp. supersolution, solution) of \eqref{eq def viscosity}.
\end{proposition}

\subsection{Weak KAM solutions and Aubry-Mather theory.}\label{sec weak kam}

We assume $H:T^*M\to \mathbb R$ is a continuous Hamiltonian satisfying
\begin{itemize}
\item [(H1)] (Convexity) $H(x,p)$ is convex in $p$ for all $x\in M$.\smallskip
\item [(H2)] (Superlinearity) $\lim\limits_{\|p\|_x\to+\infty}H(x,p)/\|p\|_x=+\infty$.
\end{itemize}

Let $L:TM\to\mathbb R$ be the convex conjugate function of $H$, i.e.,
\[L(x,v):=\sup_{p\in T^*_xM}\big(p( v)-H(x,p)\big),\quad (x,v)\in TM.
\]
It is well known that the Lagrangian $L$ is a continuous function on $TM$ and it is convex and superlinear in $v$.
The Fenchel inequality is a direct consequence of this definition:
\begin{equation}\label{fenchel}
L(x,v)+H(x,p)\geqslant p(v), \quad \hbox{for all $(x,v,p)\in M\times T_x M \times T^*_x M$.}
\end{equation}

Moreover, it can be proven that $H$ is itself the convex conjugate of $L$, i.e.,
\begin{equation}\label{convconj}
H(x,p) = \sup_{v\in T_x M} \big(p( v)-L(x,v)\big)
\qquad\hbox{for all $(x,p)\in T^*M$}.
\end{equation}

Let $c_0\in \R$ denote the critical constant defined as follows:
\begin{equation}\label{critical}
c_0=\min\{c\in\mathbb R:\ H(x,D_x u)=c\ \  \hbox{in $M$}\ \ \ \textrm{admits\ subsolutions}\}.
\end{equation}
We present here some facts that we will need about the {\em critical equation}, i.e.,
\begin{equation}\label{e0}\tag{HJ$_0$}
  H(x,D_x u)=c_0 \qquad\hbox{in $M$}.
\end{equation}

Solutions, subsolutions and supersolutions of \eqref{e0} will be termed {\em critical} in the sequel.

Due to the convex character of $H$, the following holds, see for instance \cite{barles, FathiSurvey}.

\begin{prop}\label{prop when G convex}
Let $u\in C(M)$. The following properties hold:
\begin{itemize}
\item[\em (i)] if $u$ is the pointwise supremum (respectively, infimum) of a family of subsolutions (resp., supersolutions) to \eqref{e0}, then $u$ is a subsolution (resp., supersolution) of \eqref{e0};\smallskip
\item[\em (ii)] if $u$ is the pointwise infimum of a family of equi-Lipschitz subsolutions to \eqref{e0}, then $u$ is a Lipschitz subsolution  of \eqref{e0};\smallskip
\item[\em (iii)] if $u$ is a convex combination of a family of equi-Lipschitz subsolutions to \eqref{e0}, then $u$ is a Lipschitz subsolution of \eqref{e0}.
\end{itemize}
\end{prop}

More precisely, items (ii) and (iii) above require the convexity of $H$ in the momentum, while item (i) is a general fact.

Since we are assuming $H$ to be superlinear (hence coercive, which is enough), we also have the following characterization of critical subsolutions, see for instance \cite{barles,FathiSurvey}.

\begin{prop}\label{prop equivalence}
The following are equivalent facts:
\begin{itemize}
\item[\em (i)] $v$ is a viscosity subsolution of \eqref{e0};\smallskip
\item[\em (ii)] $v$ is Lipschitz continuous and an almost everywhere subsolution of \eqref{e0}, i.e.,
\[
H(x,D_x v)\leqslant c_0\qquad\hbox{for a.e. $x\in M$.}
\]
\end{itemize}
Moreover, the set of viscosity subsolutions of  \eqref{e0} is equi-Lipschitz,  with
$\kappa_{c_0}:=\sup\{ \| p\|_x\,:\,H(x,p)\leq c_0\}$ as a common Lipschitz constant.
\end{prop}

For every $t>0$, we define the minimal action function $h_t:M\times M\to\mathbb R$ as
\[h_t(x,y)=\inf_{\gamma}\int_{-t}^0\big[L\big(\gamma(s),\dot{\gamma}(s)\big)+c_0\big]ds,\]
where $\gamma:[-t,0]\to M$ is taken among all absolutely continuous curves\footnote{In the paper, even if not explicitly stated, all curves considered are at least absolutely continuous.}   satisfying $\gamma(-t)=x$ and $\gamma(0)=y$.
The {\em Peierls barrier} is the function $h:M\times M\to\R$
defined by
\begin{equation}\label{def h}
h(x,y):=\liminf_{t\to +\infty} h_t(x,y).
\end{equation}
It satisfies the following properties, see for instance \cite{DZ10}:

\begin{prop}\label{prop h}\
\begin{itemize}
\item[\em (i)] The Peierls barrier
 $h$ is finite valued and Lipschitz
continuous.\smallskip
\item[\em (ii)]
If $v$ is a critical subsolution, then
$$
\qquad v(x)-v(y)\leqslant h(y,x),
\quad
v(x)-v(y)\leqslant h_t(y,x)
\qquad
\text{ for every $x,y\in M$ and $t>0$}.
$$
\item[\em (iii)] For every fixed $y\in
    M$, the function $h(y,\cdot)$ is a critical solution.\smallskip
\item[\em (iv)]For every fixed $y\in
    M$, the function  $-h(\cdot,y)$ is a critical subsolution.
\end{itemize}
\end{prop}

The projected {\em Aubry set} $\A$ is the closed set defined by
\[
 \A:=\{y\in M\,:\,h(y,y)=0\,\}.
\]
The following holds, see \cite{FathiSurvey, FS}:

\begin{theorem}\label{thm outA}
There exists a critical subsolution $v$ which is both strict and of class {\rm C}$^\infty$ in $M\setminus\A$, meaning that
\[
 H(x,D_x  v)<c_0\quad\text{ for every $x\in M\setminus\A$.}
\]
In particular, the projected Aubry set $\A$ is nonempty.
\end{theorem}

The last assertion directly follows from the definition of $c_0$, see \eqref{critical}.

\begin{prop}\label{prop AC}
If $u$ and $v$ are respectively a sub and supersolution such that $u\leqslant  v$ on $\mathcal A$, then $u\leqslant  v$ on the whole of $M$. In particular, $\mathcal A$ is a uniqueness set for (\ref{e0}), meaning that if two solutions coincide on $\mathcal A$, then they are equal.
\end{prop}

We will say that a Borel probability measure $\tilde{\mu}$ on $TM$ is {\em closed} if it satisfies the following conditions:
\begin{itemize}
\item [(a)] $\displaystyle \int_{TM}\|v\|_x\,d\tilde{\mu}(x,v)<+\infty$;\smallskip
\item [(b)] for all function $f\in C^1(M)$, we have $\displaystyle\int_{TM}D_xf(v)\,d\tilde{\mu}(x,v)=0$.\smallskip
\end{itemize}
We will denote by $\mathscr P_0$ the set of such measures.

We will furthermore denote by  $\PP_\ell$ the family of probability measures $\tilde\mu$ that satisfy condition (a) above.
The inclusions $\PP_0\subset \PP_\ell \subset \big(C_\ell(TM)\big)'$ hold. We will endow $ \PP_\ell $ with the weak-$*$ topology coming from the dual $\big(C_\ell(TM)\big)'$. We refer the reader to \cite{ItuCon} for more details on these families of measures.

\begin{theorem}\label{Mather}
The following holds
\[
\min_{\tilde{\mu}\in \PP_0}\int_{TM}L(x,v) \, d\tilde{\mu}=-c_0.
\]
\end{theorem}

Measures realizing the above minimum are called {\em Mather measures} for $L$. We denote by $\widetilde{\mathfrak M}$ the set of all Mather measures. This set is compact. The {\em Mather set} and the {\em projected Mather set} are defined as follows:
\[
\widetilde{\mathcal M}:=\overline{\bigcup_{\tilde{\mu}\in\widetilde{\mathfrak M}}\textrm{supp}(\tilde{\mu})},
\qquad
\M:=\pi\big(\Mtilde\big).
\]
These sets are also compact, see \cite{FathiSurvey} for a proof in the regular case.\footnote{For the present nonregular case, a proof of this can be found in Appendix A in the ArXiv version of \cite{Da4}.} Furthermore, the following holds, see \cite[Proposition 3.13]{Z}\label{AinM} for a proof in the nonregular case.

\begin{theorem}\label{thm M in A}
The following inclusion holds: $\mathcal M\subseteq \mathcal A$.
\end{theorem}

\begin{remark}\label{oss small A}
In the example of a mechanical Hamiltonian, i.e., $H(x,p)=\|p\|_x^2/2+V(x)$, it is well known that $c_0=\max_M V$,
$\A=\{y\in M\,:\,V(y)=\max_M V \}$ and the Mather measures are convex combinations of delta Diracs concentrated at points
$(y,0)$ with $y\in\A$, so that $\M$ is also equal to $\{y\in M\,:\,V(y)=\max_M V \}$.
\end{remark}

We conclude this paragraph by a technical lemma that will be of crucial use (see \cite[Theorem 2-4.1.3.]{ItuCon}).
\begin{lemma}\label{technical}
Let
$a\in \R$. The set $\{\tilde \mu \in \PP_\ell \, : \, \ \int_{TM} L(x,v)d\tilde\mu(x,v)\leqslant   a\}$ is compact in $\PP_\ell$.
\end{lemma}

Note that, as $L$ is bounded below, the quantity $\int_{TM} L(x,v)d\tilde\mu(x,v)\in \R\cup\{+\infty\}$ is well defined for any Borel probability measure. We also remark that in \cite[Theorem 2-4.1.3.]{ItuCon} the result is proved for a particular subclass of measures, however the proof makes no use of this fact and proves the above result.

\subsection{Hamiltonians depending on the unknown function}\label{sec contact H}
We recall here known results that can be found in \cite{n2,NWY} and the references therein.
In this section, we consider a continuous Hamiltonian $G:T^*M\times\mathbb R\to\mathbb R$ which satisfies the following conditions
\begin{itemize}
\item [\textbf{(G1)}] (Lipschitz in $u$) $u\mapsto G(x,p,u)$ is $K$-Lipschitz continuous for some $K>0$, uniformly in $(x,p)\in T^*M$;\smallskip
\item [\textbf{(G2)}] (Convexity in $p$) $p\mapsto G(x,p,u)$ is convex for each $(x,u)\in M\times\mathbb R$;\smallskip
\item [\textbf{(G3)}] (Superlinearity in $p$) $p\mapsto G(x,p,u)$ is superlinear for each $(x,u)\in M\times\mathbb R$.\smallskip
\end{itemize}

\begin{remark}\label{oss coercive}
The results of this paper keep holding even when the superlinearity condition (G3) is weakened in favor of a simple coercivity. For instance,
Theorems \ref{thm1} and \ref{thm2} can be easily generalized to this setting. Indeed, since we are dealing there with a family of equi-bounded, and hence equi-Lipschitz, solutions, see Lemma \ref{s3}, we could employ the usual trick of modifying $G$ outside a compact subset of $T^*M\times\R$ to make it superlinear. This cannot be done in other parts of the paper since we are dealing with families of solutions that are neither equi-bounded nor equi-Lipschitz in general. And even when they are, as in Section \ref{sec model converging}, this needs to be proved.
In fact, this is the core of the analysis performed in Section \ref{sec model converging}, which takes advantage of the fact that the Lagrangian associated with the Hamiltonian via the Fenchel duality is finite-valued. This is no longer true in the purely coercive case, even though the difficulties arising could be handled by showing that all the minimizing curves that come into play in our analysis are indeed supported on the set where the Lagrangian is finite. Yet, we believe that treating this more general case would bring additional technicalities that would have the effect of hiding the ideas at the base of this work. We prefer to leave the coercive case to a possible future work.
\end{remark}

Let $L_G:TM\times \R\to\mathbb R$ be the convex conjugate function of $G$, i.e.,
\[L_G(x,v,u):=\sup_{p\in T^*_xM}\big(p( v)-G(x,p,u)\big).\]
Then it can be proven that $L_G$ verifies similar properties (see \cite[Lemma 4.1]{V3} with easy adaptations):
\begin{itemize}
\item [\textbf{(L1)}]  $u\mapsto L_G(x,v,u)$ is $K$-Lipschitz continuous uniformly in $(x,v)\in TM$;\smallskip
\item [\textbf{(L2)}]  $v\mapsto L_G(x,v,u)$ is convex for each $(x,u)\in M\times\mathbb R$;\smallskip
\item [\textbf{(L3)}]  $v\mapsto L_G(x,v,u)$ is superlinear for each $(x,u)\in M\times\mathbb R$;\smallskip
\end{itemize}

\begin{definition}\rm
Let $G:T^*M\times\mathbb R\to\mathbb R$ be a Hamiltonian satisfying (G1-3) and let $L_G:TM\times\mathbb R\to\mathbb R$ be the associated Lagrangian.  Let $c\in\mathbb R$. A function $u\in C(M)$ satisfying the following two  properties is called a {\em backward} (resp. {\em forward}) {\em weak KAM solution} of
\begin{equation}\label{HJE}
  G\big(x,D_x u,u(x)\big)=c  \quad \hbox{in $M$}.
\end{equation}
\begin{itemize}
\item [(1)] For each absolutely continuous curve $\gamma:[t',t]\rightarrow M$, we have
\begin{equation*}
  u\big(\gamma(t)\big)-u\big(\gamma(t')\big)\leqslant  \int_{t'}^{t}\Big[L_G\Big(\gamma(s),\dot \gamma(s),u\big(\gamma(s)\big)\Big)+c\Big]ds.
\end{equation*}
The above condition reads as {\em $u$ is dominated by $L_G+c$} and will be denoted by $u\prec L_G+c$.\smallskip

\item [(2)] For each $x\in M$, there exists an absolutely continuous curve $\gamma_-:(-\infty,0]\rightarrow M$ (resp. $\gamma_+:[0,+\infty)\to M$) with $\gamma_-(0)=x$ (resp. $\gamma_+(0)=x$) such that
\begin{align*}
  &u(x)-u\big(\gamma_-(t)\big)=\int_t^0\Big[L_G\Big(\gamma_-(s),\dot \gamma_-(s),u\big(\gamma_-(s)\big)\Big)+c\Big]ds,\quad \forall t<0
  \\ &\textrm{\Big(resp.}\ u(\gamma_+(t))-u(x)=\int_0^t\Big[L_G\Big(\gamma_+(s),\dot \gamma_+(s),u\big(\gamma_+(s)\big)\Big)+c\Big]ds,\quad \forall t>0\textrm{\Big).}
\end{align*}
The curves satisfying the above equality are called {\em $(u,L_G,c)$-calibrated curves.}
\end{itemize}
\end{definition}

By \cite[Appendix D]{NWY} and \cite[Appendix A]{n2} we have
\begin{lemma}\
\begin{itemize}
\item[\em (i)]
If $u\in C(M)$ is a backward weak KAM solution of (\ref{HJE}), it is a viscosity solution of (\ref{HJE}).\smallskip
\item[\em (ii)]  The function $w\in C(M)$ is a viscosity subsolution of (\ref{HJE}) if and only if  $w\prec L_G+c$. The latter is also equivalent to $w$ being Lipschitz continuous and verifying $G\big(x,D_x w, w(x)\big)\leqslant   c$ for almost every $x\in M$.
\end{itemize}
\end{lemma}

%

We then give an approximation result in this setting. It is an easy consequence of \cite[Theorem 8.5]{FM}.

\begin{theorem}\label{approx}
Assume $G : T^*M\times \R \to \R$ is a continuous Hamiltonian verifying (G2). Let $w : M\to \R$ be a Lipschitz function verifying $G\big(x,D_x w, w(x)\big)\leqslant   c$ for almost every $x\in M$.
Then, for every $\varepsilon>0$, there is a $w_\eps\in C^\infty(M)$ such that
$\|w-w_\varepsilon\|_\infty <\varepsilon$ and
\[
G\big(x,D_x w_\varepsilon, {w}(x)\big)\leqslant   c+\varepsilon,
\quad
G\big(x,D_x w_\varepsilon, w_\varepsilon(x)\big)\leqslant   c+\varepsilon
\qquad\hbox{for all $x\in M$}.
\]
\end{theorem}

Let $a\in C(M)$. Assume there are two points $x_1$ and $x_2$ such that $a(x_1)>0$ and $a(x_2)<0$. Let $c\in\mathbb R$ and
consider the equation
\begin{equation}\label{VH0}
  a(x)u(x)+H(x,D_x u)=c  \qquad\hbox{in $M$}.
\end{equation}
This is a particular case of the previous setting for $G(x,p,u) = a(x)u+H(x,p)$. In this case, the associated Lagrangian is $L_G(x,v,u) = L(x,v) -a(x)u$.
The following implicit Lax-Oleinik semigroup $(T^-_t)_{t\geqslant   0}$ is a well defined semigroup of operators $T^-_t:C(M)\to C(M)$ that verify, for all $\varphi \in C(M)$,
\begin{eqnarray}\label{T-}
  T^-_t\varphi(x)=\inf_{\gamma(t)=x} \left\{\varphi\big(\gamma(0)\big)+\int_0^t\bigg[L\big(\gamma(\tau),\dot{\gamma}(\tau)\big)-a\big(\gamma(\tau)\big)T^-_\tau\varphi\big(\gamma(\tau)\big)+c\bigg]{d}\tau\right\},
\end{eqnarray}
where the infimum is taken among absolutely continuous curves $\gamma:[0,t]\rightarrow M$ with $\gamma(t)=x$. A similar property defines and characterizes the forward semigroup $(T^+_t)_{t>0}$, where, for all $\varphi \in C(M)$,
\begin{eqnarray}\label{T+}
  T^+_t\varphi(x)=\sup_{\gamma(0)=x} \left\{\varphi\big(\gamma(t)\big)-\int_0^t\bigg[L\big(\gamma(\tau),\dot{\gamma}(\tau)\big)-a\big(\gamma(\tau)\big)T^+_{t-\tau}\varphi\big(\gamma(\tau)\big)+c\bigg]{d}\tau\right\}.
\end{eqnarray}
\begin{lemma}\label{T-T+}\cite[Appendix A]{n2}
If $u_0$ is a subsolution of (\ref{VH0}), then
\[T^+_tu_0\leqslant  u_0\leqslant  T^-_tu_0.\]
If $u_0$ is a strict subsolution\footnote{Meaning that $c$ can be replaced with $c-\varepsilon$ for some $\varepsilon>0$.} of (\ref{VH0}), then
\[T^+_tu_0<u_0<T^-_tu_0.\]
\end{lemma}

\begin{lemma}\label{u-v-}\cite[Proposition 2.9, Proposition 3.5, Lemma 5.4]{n2}
If $u_0$ is a subsolution of (\ref{VH0}), then the limit
\[u_-:=\lim_{t\to+\infty}T^-_tu_0\quad (resp.\ v_+:=\lim_{t\to+\infty}T^+_tu_0)\]
exists, and is a viscosity solution (resp. forward weak KAM solution) of (\ref{VH0}). In addition,
\[v_-:=\lim_{t\to+\infty}T^-_tv_+\]
is also a viscosity solution of (\ref{VH0}), and there is a point $x_0\in M$ such that $v_-(x_0)=v_+(x_0)$.

If $u_0$ is a strict subsolution of (\ref{VH0}), then $u_-$ is the maximal solution of (\ref{VH0}), and $v_-$ is the minimal solution of (\ref{VH0}).
\end{lemma}

\section{ General convergence/divergence results}\label{sec general}

In this section, we consider a continuous Hamiltonian $G:T^*M\times\mathbb R\to\mathbb R$ which satisfies the following conditions:
\begin{itemize}
\item [\textbf{(G1)}] (Lipschitz in $u$) $u\mapsto G(x,p,u)$ is $K$-Lipschitz continuous uniformly in $(x,p)\in T^*M$ for some $K>0$;\smallskip
\item [\textbf{(G2)}] (Convexity in $p$) $v\mapsto G(x,p,u)$ is convex for each $(x,u)\in M\times\mathbb R$;\smallskip
\item [\textbf{(G3)}] (Superlinearity in $p$) $p\mapsto G(x,p,u)$ is superlinear for each $(x,u)\in M\times\mathbb R$;\smallskip
\item [\textbf{(G4)}] (Modulus continuity near $u=0$) The partial derivative $\frac{\partial G}{\partial u}(x,p,0)$ exists. For every compact subset $S\subset TM$, we can find a modulus of continuity\footnote{A modulus of continuity is a nondecreasing function $\eta : (0,+\infty) \to (0,+\infty)$ such that $\eta(s) \to 0$ as $s\to 0$.  } $\eta_S$ such that
    \[\bigg|G(x,p,u)-G(x,p,0)-\frac{\partial G}{\partial u}(x,p,0)u\bigg|\leqslant  |u|\eta_S(|u|),\quad \forall (x,p)\in S.
    \]
\end{itemize}
The dependence of $G$ in the  $u$-variable is nonlinear, in general, as in \cite{CFZZ}, but we do not make any global monotonicity assumption.
As established in  \cite[Lemma 4.1]{V3} the associated Lagrangian function $L_G:TM\times \R\to\mathbb R$ defined by
\[L_G(x,v,u):=\sup_{p\in T^*_xM}\big(p( v)-G(x,p,u)\big),\]
has similar properties:

\begin{itemize}
\item [\textbf{(L1)}] (Lipschitz in $u$) $u\mapsto L_G(x,v,u)$ is $K$-Lipschitz continuous uniformly in $(x,v)\in TM$;\smallskip
\item [\textbf{(L2)}] (Convexity in $v$) $v\mapsto L_G(x,v,u)$ is convex for each $(x,u)\in M\times\mathbb R$;\smallskip
\item [\textbf{(L3)}] (Superlinearity in $v$) $v\mapsto L_G(x,v,u)$ is superlinear for each $(x,u)\in M\times\mathbb R$;\smallskip
\item [\textbf{(L4)}] (Modulus continuity near $u=0$) The partial derivative $\frac{\partial L_G}{\partial u}(x,v,0)$ exists. For every compact subset $S\subset TM$, we can find a modulus of continuity $\eta_S$ such that
    \[\bigg|L_G(x,v,u)-L_G(x,v,0)-\frac{\partial L_G}{\partial u}(x,v,0)u\bigg|\leqslant  |u|\eta_S(|u|),\quad \forall (x,v)\in S.
    \]
    \end{itemize}
    The last condition is easier to state with the Lagrangian function as it involves Mather measures defined on $TM$:
    \begin{itemize}
\item [\textbf{(L5)}] (Non-degeneracy condition) For all Mather measures $\tilde \mu$ of $(x,v) \mapsto L_G(x,v,0)$, we have
\[
\int_{TM} \frac{\partial L_G}{\partial u}(x,v,0)\,d\tilde{\mu}<0.
\]
\end{itemize}
\begin{remark}\rm
(1) It is classical in convex analysis that for all $(x,v)\in TM$ there is $(x,p)\in T^*M$ verifying $L_G(x,v,0)+G(x,p,0) =p(v)$.
It is proved  in \cite[Lemma 4.1]{V3}  that if $(x,p)\in T^*M$ and $(x,v)\in TM$ verify the previous formula  then $\frac{\partial L_G}{\partial u} (x,v,0) = -\frac{\partial G}{\partial u} (x,p,0)$.
Therefore, an equivalent formulation of (L5) is
\begin{itemize}
\item [\textbf{(G5)}] (Non-degeneracy condition) For all Mather measures $\tilde \mu$ of $(x,v) \mapsto L_G(x,v,0)$,
\[\int_{TM} \frac{\partial G}{\partial u}(x,p_{(x,v)},0)\,d\tilde{\mu}>0.
\]
where $p_{(x,v)}$ is chosen so to satisfy $L_G(x,v,0)+G(x,p_{(x,v)},0) =p_{(x,v)}(v)$ .
\end{itemize}
\indent (2) A result of Ma\~ n\' e (\cite{Mane}) asserts that a generic Hamiltonian $H$ has a unique Mather measure.\footnote{It has been conjectured  by Ma\~ n\' e that for generic Hamiltonians $H$ this unique Mather measure is supported on a closed curve.} Hence our condition of integral type is quite loose for such a generic $H$. We also refer the reader to Remark \ref{oss small A} for an explicit example.
 \end{remark}

Consider
\begin{equation}\label{E}\tag{E$_\lambda$}
  G\big(x,D_xu,\lambda u(x)\big)=c_0 \qquad\hbox{in $M$},
\end{equation}
and the limit equation
\begin{equation}\label{E0}\tag{${\textrm{E}}_0$}
  G(x,D_xu,0)=c_0 \qquad\hbox{in $M$}.
\end{equation}
Here we denote by $c_0$ the critical value of $H(x,p):=G(x,p,0)$. We denote by $\widetilde{\mathfrak M}$ (resp. $\widetilde{\mathcal M}$) the set of all Mather measures (resp. the Mather set) corresponding to $H$. In the sequel, we will always assume that $\lambda$ belongs to the interval $(0,1)$.

The main theorems of this section are the following.
\begin{theorem}\label{thm1}
Let conditions (L1--5) be in force. Let us assume that there exist an equi-bounded family $(u_\lambda)_{\lambda\in (0,\lambda_0)}$  of solutions to  (\ref{E}), for some $\lambda_0 \in (0,1)$.  Then the functions  $u_\lambda$ uniformly converge in $M$, as $\lambda\to 0^+$, to a solution $u_0$ of the critical equation \eqref{E0}. Moreover, $u_0$ is the largest subsolution $w$ of (\ref{E0}) satisfying
\begin{equation}\label{subcon}\tag{S}
  \int_{TM}w(x)\frac{\partial L_G}{\partial u}(x,v,0)\,d\tilde{\mu}(x,v)\geqslant  0,\quad \forall \tilde{\mu}\in\widetilde{\mathfrak M}.
\end{equation}
\end{theorem}

Under an additional pointwise condition on $L_G(\cdot,\cdot,0)$ on the Mather set, the limit critical solution $u_0$ can be characterized as follows.

\begin{theorem}\label{thm2}
Let conditions (L1--5) be in force, and let us furthermore assume that $\frac{\partial L_G}{\partial u}(x,v,0)\leqslant  0$ for all $(x,v)\in\widetilde{\mathcal M}$. Then the function $u_0=\lim\limits_{\lambda\to 0^+} u_\lambda$ obtained in Theorem \ref{thm1} above can be characterized as follows:
\[
u_0(x)=\inf_{\tilde{\mu}\in\widetilde{\mathfrak M}}\frac{\int_{TM} h(y,x)\frac{\partial L_G}{\partial u}(y,v,0)d\tilde\mu(y,v)}{\int_{TM} \frac{\partial L_G}{\partial u}(y,v,0)d\tilde\mu(y,v)},\qquad {x\in M},
\]
where $h(y,x)$ is the Peierls barrier defined in \eqref{def h}.
\end{theorem}

\begin{remark}\label{oss picky}
Let us stress that the previous theorems hold with same proofs even if  $(u_\lambda)_{\lambda\in (0,\lambda_0)}$ is replaced by
any family $(u_\lambda)_{\lambda\in\Lambda}$ of solutions to \eqref{E}, where $\Lambda$ is a subset of $(0,1)$ having $0$ as accumulation point. In particular, this holds for $\Lambda:=(\lambda_n)_n$ with $\lambda_n \to 0^+$ as $n\to +\infty$.
\end{remark}

Concerning the asymptotic behavior of other possible families of solutions to \eqref{E}, we have the following trichotomy result.

\begin{theorem}\label{thm Maxime bis}
Let conditions (L1--5) be in force, and let  $u_0=\lim\limits_{\lambda\to 0^+} u_\lambda$ be the critical solution obtained in Theorem \ref{thm1}.
There exist $\varphi:(0,1)\to\mathbb (-\infty,+\infty]$,  $\psi:(0,1)\to\mathbb [-\infty,+\infty)$ and $\theta:(0,1)\to\mathbb R$ with
\[\lim_{\lambda\to 0}\psi(\lambda)=-\infty, \quad\quad \lim_{\lambda\to 0}\varphi(\lambda)=+\infty,\quad \quad \lim_{\lambda\to 0}\theta(\lambda)=0\]
such that, if $v_\lambda$ is a solution of (\ref{E}) for some $\lambda>0$, then either one of the following alternatives occurs:
\begin{itemize}
\item[\em (i)] $v_\lambda\leqslant  \psi(\lambda)$;\smallskip
\item[\em (ii)] $v_\lambda\geqslant  \varphi(\lambda)$;\smallskip
\item[\em (iii)] $\|v_\lambda-u_0\|_\infty\leqslant  \theta(\lambda)$.
\end{itemize}
\end{theorem}

We start our analysis by remarking that the solutions  $(u_\lambda)_{\lambda\in (0,\lambda_0)}$ are equi-Lipschitz continuous. In the remainder of the section, we will denote by $C>0$ the following constant
\[
C:=\sup_{\lambda\in(0,\lambda_0)} \|u_\lambda\|_\infty.
\]

\begin{lemma}\label{s3}
The bounded family $(u_\lambda)_{\lambda\in (0,\lambda_0)}$ is equi-Lipschitz continuous.
\end{lemma}
\begin{proof}
For each $x,y\in M$, we denote by $ d :=d(x,y)$ the distance between them. Take a geodesic $\zeta:[0, d ]\rightarrow M$ satisfying $\zeta(0)=x$ and $\zeta( d )=y$ with constant speed $\|\dot{\zeta}\|_\zeta=1$. Denote $L_\lambda(x,v):=L_G(x,v,\lambda u_\lambda)$. Since $u_\lambda\prec L_\lambda+c_0$, we have
\begin{align*}
u_\lambda(y)-u_\lambda(x)&\leqslant  \int_0^{ d } \bigg[L_G\Big(\zeta(s),\dot{\zeta}(s),\lambda u_\lambda\big(\zeta(s)\big)\Big)+c_0\bigg]ds
\\ &\leqslant  \bigg(\max_{x\in M,\|v\|_x\leqslant  1}|L_G(x,v,0)|+\lambda_0KC +c_0\bigg) d(x,y).
\end{align*}
The assertion follows by exchanging the role of $x$ and $y$.
\end{proof}

From the fact that $u_\lambda$ is Lipschitz for every fixed $\lambda\in (0,\lambda_0)$ we deduce the following fact.

\begin{lemma}
Let $\lambda \in (0,\lambda_0)$. Then
$$\int_{TM} L_G\big(x,v,\lambda u_\lambda(x)\big) d\tilde \mu(x,v)\geqslant   -c_0\qquad\hbox{for all $\tilde \mu \in \mathscr P_0$.}
$$
\end{lemma}
\begin{proof}
Pick  $\tilde \mu \in \mathscr P_0$  and choose $\varepsilon>0$. By applying Theorem \ref{approx} to the Hamiltonian $(x,p)\mapsto G\big(x,p,\lambda u_\lambda(x)\big)$ and by choosing $w:=u_\lambda$, we infer that  there exists $w_\varepsilon\in C^\infty(M)$ such that $G\big(x, D_xw_\varepsilon , \lambda u_\lambda(x)\big)\leqslant   c_0+\varepsilon$ for all $x\in M$.
By definition of $L_G$ we have that
\[
G\big(x, D_xw_\varepsilon , \lambda u_\lambda(x)\big)+L_G\big(x,v,\lambda u_\lambda(x)\big) \geqslant   D_xw_\varepsilon(v)
\qquad\hbox{for all $(x,v)\in TM$}.
\]
By integrating this inequality with respect to $\tilde\mu$ we get
\begin{multline*}
0 = \int_{TM} D_xw_\varepsilon(v)\,  d\tilde \mu(x,v)\\
\leqslant    \int_{TM} \Big(L_G\big(x,v,\lambda u_\lambda(x)\big) +G\big(x, D_xw_\varepsilon , \lambda u_\lambda(x)\big)\Big)d\tilde \mu(x,v)\\
\leqslant    \int_{TM} \Big(L_G\big(x,v,\lambda u_\lambda(x)\big) +c_0+\varepsilon \Big)d\tilde \mu(x,v) =  \int_{TM} L_G\big(x,v,\lambda u_\lambda(x)\big)  d\tilde \mu(x,v)+c_0+\varepsilon.
\end{multline*}
The result follows letting $\varepsilon \to 0^+$.
\end{proof}

By the Arzel\'a-Ascoli Theorem and Lemma \ref{s3}, any sequence $(u_{\lambda_n})_n$ with $\lambda_n\to 0^+$ admits a subsequence which
uniformly converges to a continuous function $u^*$. By the stability of viscosity solution, see Proposition \ref{prop stability}, $u^*$ is a solution of (\ref{E0}). In the following, we are going to show the uniqueness of the possible limit $u^*$, thus establishing the convergence result. Define $\mathcal S$  the set of all subsolutions $w$ of (\ref{E0}) satisfying condition (\ref{subcon}). We define
\begin{equation}\label{u0}
u_0(x):=\sup_{w\in \mathcal S}w(x).
\end{equation}
A priori, $\mathcal S$ may be empty, and,  even if $\mathcal S$ is not empty, $u_0$ might be $+\infty$. Both these circumstances will be excluded under the hypotheses of Theorem \ref{thm1}.

\begin{lemma}\label{<u0}
Any accumulation point $u^*$ of the family $(u_\lambda)_{\lambda\in(0,\lambda_0)}$ as $\lambda\to 0^+$ satisfies
\[\int_{TM}\frac{\partial L_G}{\partial u}(x,v,0)u^*(x)d\mu(x,v)\geqslant  0,\quad \forall \tilde{\mu}\in\widetilde{\mathfrak M}.\]
In particular, $\mathcal S\not=\emptyset$  and $u^*\leqslant  u_0$.
\end{lemma}
\begin{proof}
Recall that $C $ is the uniform bound of $(u_\lambda)_{\lambda\in(0,\lambda_0)}$. Let $\tilde{\mu}\in\widetilde{\mathfrak M}$. For $\lambda\in (0,\lambda_0)$  we have
\begin{align*}
-c_0&\leqslant  \int_{TM} L_G\big(x,v,\lambda u_\lambda(x)\big)d\tilde{\mu}
\\ &\leqslant  \int_{TM}\bigg[L_G(x,v,0)+\lambda\frac{\partial L_G}{\partial u}(x,v,0)u_\lambda(x)+\lambda C \eta_{\widetilde{\mathcal M}}(\lambda C )\bigg]d\tilde{\mu}
\\ &=-c_0+\int_{TM}\bigg[\lambda\frac{\partial L_G}{\partial u}(x,v,0)u_\lambda(x)+\lambda C \eta_{\widetilde{\mathcal M}}(\lambda C )\bigg]d\tilde{\mu},
\end{align*}
which implies
\[\int_{TM}\frac{\partial L_G}{\partial u}(x,v,0)u_\lambda(x)d\tilde{\mu}\geqslant  -C \eta_{\widetilde{\mathcal M}}(\lambda C ).\]
The conclusion follows by sending $\lambda\to 0^+$.
\end{proof}

In what follows, we will use the notation $L_\lambda(x,v):=L_G(x,v,\lambda u_\lambda)$.

\begin{lemma}\label{gameq}
For $x\in M$ and $\lambda\in (0,\lambda_0)$, let $\gamma^x_\lambda:(-\infty,0]\to M$ be a $(u_\lambda,L_\lambda,c_0)$-calibrated curve with $\gamma^x_\lambda(0)=x$. Then there exists $\hat\kappa>0$, independent of $\lambda\in (0,\lambda_0)$ and of $x\in M$, such that $\gamma^x_\lambda$ is $\hat\kappa$-Lipschitz continuous for every $\lambda\in (0,\lambda_0)$ and $x\in M$.
\end{lemma}
\begin{proof}
By Lemma \ref{s3}, there is $\kappa>0$ independent of $\lambda$ such that $u_\lambda$ is $\kappa$-Lipschitz continuous. By superlinearity of $L_G$, for each $T>0$, there is $C_T\in\mathbb R$ such that
\begin{equation}\label{Ct}
  L_G(x,v,0)\geqslant  T\|v\|_x+C_T.
\end{equation}
Thus, we have for $0\geqslant   t>s$
\begin{flalign*}
\ \kappa d\big(\gamma^x_\lambda(t),\gamma^x_\lambda(s)\big)
&\geqslant  u_\lambda\big(\gamma^x_\lambda(t)\big)-u_\lambda\big(\gamma^x_\lambda(s)\big)
=\int_s^t \bigg[L_G\Big(\gamma^x_\lambda(\tau),\dot{\gamma}^x_\lambda(\tau),\lambda u_\lambda\big(\gamma^x_\lambda(\tau)\big)\Big)+c_0\bigg]d\tau &&
\\
&\geqslant  \int_s^t \big((\kappa+1)\|\dot{\gamma}^x_\lambda(\tau)\|_{\gamma^x_\lambda(\tau)}+C_{\kappa+1}\big)d\tau+(c_0-\lambda_0KC )(t-s) &&
\\
&\geqslant  (\kappa+1)d\big(\gamma^x_\lambda(t),\gamma^x_\lambda(s)\big)+(C_{\kappa+1}+c_0-\lambda_0KC )(t-s), &&
\end{flalign*}
which implies
\[d\big(\gamma^x_\lambda(t),\gamma^x_\lambda(s)\big)\leqslant  (\lambda_0KC -c_0-C_{\kappa+1})(t-s).\]
The proof is now complete.
\end{proof}

From now on, we denote by $\gamma^x_\lambda$ the calibrated curve considered in Lemma \ref{gameq}. By the compactness of $\widetilde{\mathfrak M}$, there are two constants $\epsilon_1>0$ and $\epsilon_2>0$ with
\begin{equation}\label{e1<e2}
  -\epsilon_2<\inf_{\tilde{\mu}\in\widetilde{\mathfrak M}}\int_{TM}\frac{\partial L_G}{\partial u}(x,v,0)d\tilde{\mu}\leqslant  \sup_{\tilde{\mu}\in\widetilde{\mathfrak M}}\int_{TM}\frac{\partial L_G}{\partial u}(x,v,0)d\tilde{\mu}<-\epsilon_1,
\end{equation}

We derive from this the following asymptotic informations on the calibrated curves $\gamma^x_\lambda$, cf. \cite[Corollary 7.4]{CFZZ}.

\begin{lemma}\label{e1e2}
There exist $\bar \lambda\in (0,\lambda_0)$ and $T_0>0$ such that, for all $\lambda\in(0,\bar \lambda)$ and for all $x\in M$, we have
\begin{equation}\label{ab}
  -\epsilon_2<\frac{1}{b-a}\int_a^b \frac{\partial L_G}{\partial u}(\gamma^x_\lambda(s),\dot{\gamma}^x_\lambda(s),0)ds<-\epsilon_1,
\end{equation}
for all $a<b\leqslant  0$ with $b-a\geqslant  T_0$.
\end{lemma}
\begin{proof}
Let us prove the {left-hand} inequality in (\ref{ab}). We argue by contradiction. Assume there is a sequence $\lambda_n\to 0$ and a sequence $b_n-a_n\to +\infty$ such that
\begin{equation}\label{>e2}
  \frac{1}{b_n-a_n}\int_{a_n}^{b_n}\frac{\partial L_G}{\partial u}(\gamma^x_{\lambda_n}(s),\dot{\gamma}^x_{\lambda_n}(s),0)ds\leqslant  -\epsilon_2.
\end{equation}
Define for all $n\in \N$ a probability measure $\tilde{\mu}_n$ on $TM$ by
\[\int_{TM}g(x,v)d\tilde{\mu}_n:=\frac{1}{b_n-a_n}\int_{a_n}^{b_n} g\big(\gamma^x_{\lambda_n}(s),\dot{\gamma}^x_{\lambda_n}(s)\big)ds,\quad \forall g\in C_c(TM).\]
Here $C_c(TM)$ is the set of all continuous functions defined on $TM$ with compact supports. By Lemma \ref{gameq}, all the measures $\tilde{\mu}_n$ have support in the compact set
\[\{(x,v)\in TM:\ \|v\|_x\leqslant  \lambda_0KC -c_0-C_{\kappa+1}\}.\]
Up to extracting a subsequence if necessary, we can assume that $\tilde{\mu}_n$ converges to $\tilde{\mu}$ in the weak-$*$ topology on $C_c(TM)$. Note that $\tilde \mu$ is also compactly supported.  For $f\in C^1(M)$, we have
\begin{align*}
\int_{TM}D_xf(v)d\tilde{\mu}_n(x,v)&=\frac{1}{b_n-a_n}\int_{a_n}^{b_n} D_{\gamma^x_{\lambda_n}(s) }f\big(\dot{\gamma}^x_{\lambda_n}(s)\big)ds
\\ &=\frac{1}{b_n-a_n}\Big(f\big(\gamma^x_{\lambda_n}(b_n)\big)-f\big(\gamma^x_{\lambda_n}(a_n)\big)\Big)\to 0,
\end{align*}
as $n\to +\infty$, which implies that $\tilde{\mu}$ is closed. Since $\gamma^x_{\lambda_n}$ is a calibrated curve, we have
\begin{align*}
&\int_{TM}\bigg[L_G\big(x,v,\lambda_n u_{\lambda_n}(x)\big)+c_0\bigg]d\tilde{\mu}_n(x,v)
\\ &=\frac{1}{b_n-a_n}\int_{a_n}^{b_n} \bigg[L_G\Big(\gamma^x_{\lambda_n}(s),\dot{\gamma}^x_{\lambda_n}(s),\lambda_n u_{\lambda_n}\big(\gamma^x_{\lambda_n}(s)\big)\Big)+c_0\bigg]ds
\\ &=\frac{1}{b_n-a_n}\Big(u_{\lambda_n}\big(\gamma^x_{\lambda_n}(b_n)\big)-u_{\lambda_n}\big(\gamma^x_{\lambda_n}(a_n)\big)\Big).
\end{align*}
Since $u_\lambda$ is bounded, letting $n\to +\infty$, we find that
\[\int_{TM}L_G(x,v,0)d\tilde{\mu}=-c_0.\]
Therefore, the limit $\tilde{\mu}$ is a Mather measure of $H$. By (\ref{>e2}), we get
\[\int_{TM}\frac{\partial L_G}{\partial u}(x,v,0)d\tilde{\mu}(x,v)\leqslant  -\epsilon_2,\]
which contradicts (\ref{e1<e2}). The {right-hand} side inequality in (\ref{ab}) can be proved similarly.
\end{proof}

In the following, we will denote by $\bar \lambda >0$ and $T_0>0$ the constants given by Lemma \ref{e1e2}.

\begin{lemma}\label{>T0}
Let $\lambda\in(0,\bar \lambda)$ and $x\in M$. The following holds:
\begin{itemize}
\item[\em (i)] for any $t\in (-\infty,-T_0]$, we have
\[\epsilon_2 t\leqslant  \int_t^0 \frac{\partial L_G}{\partial u}(\gamma^x_\lambda(s),\dot{\gamma}^x_\lambda(s),0)ds\leqslant  \epsilon_1 t.\]
As a consequence,\
$\displaystyle e^{\lambda\int_{-\infty}^0 \frac{\partial L_G}{\partial u}(\gamma^x_\lambda(s),\dot{\gamma}^x_\lambda(s),0)\,ds}=0.$\smallskip
\item[\em (ii)] For any $T\geqslant  T_0$, we have
\[\frac{e^{-\lambda \epsilon_2 T_0}-e^{-\lambda \epsilon_2 T}}{\lambda \epsilon_2}\leqslant  \int_{-T}^0 e^{\lambda \int_t^0\frac{\partial L_G}{\partial u}(\gamma^x_\lambda(s),\dot{\gamma}^x_\lambda(s),0)ds}dt
\leqslant
\frac{1}{\lambda \epsilon_1}+\frac{e^{\lambda K T_0}-1}{\lambda K}.\]
In particular,
\begin{equation}\label{einf}
  \frac{e^{-\lambda \epsilon_2 T_0}}{\lambda \epsilon_2}\leqslant  \int_{-\infty}^0 e^{\lambda \int_t^0\frac{\partial L_G}{\partial u}(\gamma^x_\lambda(s),\dot{\gamma}^x_\lambda(s),0)ds}dt
  \leqslant
  \frac{1}{\lambda \epsilon_1}+\frac{e^{ \lambda K  T_0}-1}{\lambda K}.
\end{equation}
\end{itemize}
\end{lemma}
\begin{proof}
Item (i) is a direct consequence of Lemma \ref{e1e2}. It remains to prove Item (ii). By Item (i) and (L1) we have
\begin{flalign*}
&
\int_{-T}^0 e^{\lambda \int_t^0\frac{\partial L_G}{\partial u}(\gamma^x_\lambda(s),\dot{\gamma}^x_\lambda(s),0)ds}dt
=
\int_{-T}^{-T_0} \!\! e^{\lambda \int_t^0\frac{\partial L_G}{\partial u}(\gamma^x_\lambda(s),\dot{\gamma}^x_\lambda(s),0)ds}dt
+
\int_{-T_0}^0 \!\!  e^{\lambda \int_t^0\frac{\partial L_G}{\partial u}(\gamma^x_\lambda(s),\dot{\gamma}^x_\lambda(s),0)ds}dt
&&
\\
&
\leqslant
\int_{-T}^{-T_0} e^{\lambda \epsilon_1 t}dt+\int_{-T_0}^0 e^{  -\lambda K t}dt
=\frac{e^{-\lambda\epsilon_1 T_0}-e^{-\lambda\epsilon_1 T}}{\lambda \epsilon_1}+\frac{e^{  \lambda K T_0}-1}{  \lambda K}
\leqslant
\frac{1}{\lambda \epsilon_1}+\frac{e^{  \lambda K T_0}-1}{  \lambda K}.
&&
\end{flalign*}
For the other side, we have
\begin{align*}
\int_{-T}^0 e^{\lambda \int_t^0\frac{\partial L_G}{\partial u}(\gamma^x_\lambda(s),\dot{\gamma}^x_\lambda(s),0)ds}dt&\geqslant  \int_{-T}^{-T_0} e^{\lambda \int_t^0\frac{\partial L_G }{\partial u}(\gamma^x_\lambda(s),\dot{\gamma}^x_\lambda(s),0)ds}dt
\\ &\geqslant  \int_{-T}^{-T_0} e^{\lambda \epsilon_2 t}dt=\frac{e^{-\lambda\epsilon_2 T_0}-e^{-\lambda\epsilon_2 T}}{\lambda \epsilon_2}.
\end{align*}
Let $T\to +\infty$, we then get (\ref{einf}).
\end{proof}

We proceed by associating to each calibrated curve a probability measure on $TM$. These probability measures will play a key role in the proof of the convergence result stated in Theorem \ref{thm1}. They can be regarded as a generalization to the case at issue of the analogous measures first introduced in \cite[formula (3.5)]{Da1}. They already appeared in this exact form in \cite{V3,Z,CFZZ}.

\begin{definition}\label{mulam}\rm
We define  probability measures $\tilde{\mu}^x_\lambda$ on $TM$ by
\[\int_{TM}f(y,v)d\tilde{\mu}^x_\lambda(y,v)=\frac{\int_{-\infty}^0 f\big(\gamma^x_\lambda(t),\dot{\gamma}^x_\lambda(t)\big)e^{\lambda \int_t^0 \frac{\partial L_G}{\partial u}(\gamma^x_\lambda(s),\dot{\gamma}^x_\lambda(s),0)ds}dt}{\int_{-\infty}^0 e^{\lambda \int_t^0 \frac{\partial L_G}{\partial u}(\gamma^x_\lambda(s),\dot{\gamma}^x_\lambda(s),0)ds}dt},\quad \forall f\in C_c(TM).\]
By (\ref{einf}), the measure $\tilde{\mu}^x_\lambda$ is well-defined for $\lambda\in (0,\bar \lambda)$.
\end{definition}

The following holds, cf. \cite[Proposition 3.6]{Da1}, \cite[Proposition 4.5]{V3}, \cite[Proposition 5.8]{Z}, \cite[Proposition 7.5]{CFZZ}.

\begin{lemma}\label{muisma}
The family $(\tilde{\mu}^x_\lambda)_{\lambda\in(0,\bar \lambda)}$ has support contained in a common compact subset of $TM$, in particular it is relatively compact in $\parts (TM)$. Furthermore, if $\tilde{\mu}^x_{\lambda_n}\weakst  \tilde{\mu}$  in $\parts (TM)$ for $\lambda_n\to 0^+$, then $\tilde\mu$ is a Mather measure.
\end{lemma}
\begin{proof}
The first part is a direct consequence of Lemma \ref{gameq}. It remains to prove that the limit $\tilde{\mu}$ is a Mather measure. We first prove that $\tilde{\mu}$ is closed. For $f\in C^1(M)$, we have, by integrating by parts,
\begin{align*}
\int_{TM}D_xf(v)d\tilde{\mu}^x_\lambda(y,v)&=\frac{\int_{-\infty}^0 \frac{d}{dt}\Big(f\big(\gamma^x_\lambda(t)\big)\Big)e^{\lambda \int_t^0 \frac{\partial L_G}{\partial u}(\gamma^x_\lambda(s),\dot{\gamma}^x_\lambda(s),0)ds}dt}{\int_{-\infty}^0 e^{\lambda \int_t^0\frac{\partial L_G}{\partial u}(\gamma^x_\lambda(s),\dot{\gamma}^x_\lambda(s),0) ds}dt}
\\ &=\frac{f(x)-\int_{-\infty}^0 f\big(\gamma^x_\lambda(t)\big)\frac{d}{dt}\Big(e^{\lambda \int_t^0 \frac{\partial L_G}{\partial u}(\gamma^x_\lambda(s),\dot{\gamma}^x_\lambda(s),0)ds}\Big)dt}{{\int_{-\infty}^0 e^{\lambda \int_t^0 \frac{\partial L_G}{\partial u}(\gamma^x_\lambda(s),\dot{\gamma}^x_\lambda(s),0)ds}dt}}.
\end{align*}
By (L1) and $e^{\lambda \int_t^0 \frac{\partial L_G}{\partial u}(\gamma^x_\lambda(s),\dot{\gamma}^x_\lambda(s),0)ds}>0$, we have
\begin{align*}
&\bigg|\int_{-\infty}^0 f\big(\gamma^x_\lambda(t)\big)\frac{d}{dt}\Big(e^{\lambda \int_t^0 \frac{\partial L_G}{\partial u}(\gamma^x_\lambda(s),\dot{\gamma}^x_\lambda(s),0)ds}\Big)dt\bigg|
\\ &=\bigg|\int_{-\infty}^0 \lambda f\big(\gamma^x_\lambda(t)\big)\frac{\partial L_G}{\partial u}(\gamma^x_\lambda(t),\dot{\gamma}^x_\lambda(t),0)e^{\lambda \int_t^0 \frac{\partial L_G}{\partial u}(\gamma^x_\lambda(s),\dot{\gamma}^x_\lambda(s),0)ds}dt\bigg|
\\ &\leqslant  \lambda\|f\|_\infty K \int_{-\infty}^0 e^{\lambda \int_t^0 \frac{\partial L_G}{\partial u}(\gamma^x_\lambda(s),\dot{\gamma}^x_\lambda(s),0)ds}dt \leqslant  K\bigg(\frac{1}{\epsilon_1}+\frac{  e^{\lambda K T_0}-1}{  K}\bigg)\|f\|_\infty,
\end{align*}
where, for the last inequality, we have used (\ref{einf}). By (\ref{einf}) again, we have
\begin{align*}
\int_{TM}D_xf(v)d\tilde{\mu}^x_\lambda(y,v)&\leqslant  \frac{\lambda \epsilon_2}{e^{-\lambda \epsilon_2 T_0}}K\bigg(\frac{1}{\epsilon_1}+\frac{  e^{\lambda K T_0}}{  K} \bigg)\|f\|_\infty\to 0,
\end{align*}
as $\lambda \to 0^+$.

We then prove that $\tilde{\mu}$ is minimizing. Since $t\mapsto u_\lambda\big(\gamma^x_\lambda(t)\big)$ is Lipschitz continuous, and $\gamma^x_\lambda$ is a $(u_\lambda,L_\lambda,c_0)$-calibrated curve, for a.e. $t<0$ we have
\[\frac{d}{dt}u_\lambda\big(\gamma^x_\lambda(t)\big)=L_G\Big(\gamma^x_\lambda(t),\dot{\gamma}^x_\lambda(t),\lambda u_\lambda\big(\gamma^x_\lambda(t)\big)\Big)+c_0.\]
Then
\begin{align*}
&\int_{TM}\big(L_G(x,v,0)+c_0\big)d\tilde{\mu}^x_\lambda(x,v)
\\ &=\frac{\int_{-\infty}^0 \big(L_G(\gamma^x_\lambda(t),\dot{\gamma}^x_\lambda(t),0)+c_0\big)e^{\lambda \int_t^0 \frac{\partial L_G}{\partial u}(\gamma^x_\lambda(s),\dot{\gamma}^x_\lambda(s),0)ds}dt}{\int_{-\infty}^0 e^{\lambda \int_t^0 \frac{\partial L_G}{\partial u}(\gamma^x_\lambda(s),\dot{\gamma}^x_\lambda(s),0)ds}dt}
\\ &=\frac{\int_{-\infty}^0 \Big(L_G\Big(\gamma^x_\lambda(t),\dot{\gamma}^x_\lambda(t),\lambda u_\lambda\big(\gamma^x_\lambda(t)\big)\Big)+c_0-\Delta_\lambda(t)\Big)e^{\lambda \int_t^0 \frac{\partial L_G}{\partial u}(\gamma^x_\lambda(s),\dot{\gamma}^x_\lambda(s),0)ds}dt}{\int_{-\infty}^0 e^{\lambda \int_t^0 \frac{\partial L_G}{\partial u}(\gamma^x_\lambda(s),\dot{\gamma}^x_\lambda(s),0)ds}dt}
\\ &= \frac{\int_{-\infty}^0 \Big(\frac{d}{dt}\Big(u_\lambda\big(\gamma^x_\lambda(t)\big)\Big)-\Delta_\lambda(t)\Big)e^{\lambda \int_t^0 \frac{\partial L_G}{\partial u}(\gamma^x_\lambda(s),\dot{\gamma}^x_\lambda(s),0)ds}dt}{\int_{-\infty}^0 e^{\lambda \int_t^0 \frac{\partial L_G}{\partial u}(\gamma^x_\lambda(s),\dot{\gamma}^x_\lambda(s),0)ds}dt},
\end{align*}
where
\[\Delta_\lambda(t)=L_G\Big(\gamma^x_\lambda(t),\dot{\gamma}^x_\lambda(t),\lambda u_\lambda\big(\gamma^x_\lambda(t)\big)\Big)-L_G\big(\gamma^x_\lambda(t),\dot{\gamma}^x_\lambda(t),0\big).\]
Similarly to the first part of the proof, we have
\begin{align*}
\frac{\int_{-\infty}^0 \frac{d}{dt}u_\lambda\big(\gamma^x_\lambda(t)\big)e^{\lambda \int_t^0 \frac{\partial L_G}{\partial u}(\gamma^x_\lambda(s),\dot{\gamma}^x_\lambda(s),0)ds}dt}{\int_{-\infty}^0 e^{\lambda \int_t^0 \frac{\partial L_G}{\partial u}(\gamma^x_\lambda(s),\dot{\gamma}^x_\lambda(s),0)ds}dt}
\leqslant  \frac{\lambda \epsilon_2}{e^{-\lambda \epsilon_2 T_0}}K\bigg(\frac{1}{\epsilon_1}+\frac{  e^{\lambda K T_0}}{  K}\bigg)C \to 0,
\end{align*}
as $\lambda \to 0^+$, where $C $ is a uniform bound on the $\|u_\lambda\|_\infty$.

Finally, to bound the error term we use (L1) to find
\begin{align*}
|\Delta_\lambda(t)|=\Big|L_G\Big(\gamma^x_\lambda(t),\dot{\gamma}^x_\lambda(t),\lambda u_\lambda\big(\gamma^x_\lambda(t)\big)\Big)-L\big(\gamma^x_\lambda(t),\dot{\gamma}^x_\lambda(t),0\big)\Big|
\leqslant  \lambda KC .
\end{align*}
Therefore, as $\lambda \to 0^+$ along the sequence $(\lambda_n)$, we conclude that $\tilde{\mu}$ is a Mather measure.\qedhere
\end{proof}

The following lemma will be crucial for the proof of the convergence result, cf. \cite[Lemma 3.7]{Da1}, \cite[Lemma 4.7]{V3}, \cite[Lemma 7.7]{CFZZ}.

\begin{lemma}\label{>u0}
Let $w$ be any subsolution of (\ref{E0}). For every $x\in M$ and $\lambda\in (0,\bar \lambda)$, we have
\[u_\lambda(x)\geqslant  w(x)-\frac{\int_{TM}w(y)\frac{\partial L_G}{\partial u}(y,v,0)d\tilde{\mu}^x_\lambda(y,v)}{\int_{TM}\frac{\partial L_G}{\partial u}(y,v,0)d\tilde{\mu}^x_\lambda(y,v)}+R_\lambda(x),\]
where $\lim\limits_{\lambda \to 0^+}R_\lambda(x)=0$.
\end{lemma}
\begin{proof}
For $\varepsilon>0$, using Theorem \ref{approx} we take $w_\varepsilon\in C^\infty(M)$ such that $\|w_\varepsilon-w\|_\infty\leqslant  \varepsilon$ and
\[G(x,D_x w_\varepsilon ,0)\leqslant  c_0+\varepsilon,\quad \forall x\in M.\]
Using the Fenchel inequality we have for all $ (x,v)\in TM$,
\begin{align*}
L_G(x,v,0)&\geqslant  L_G(x,v,0)+G(x,D_x w_\varepsilon ,0)-c_0-\varepsilon
\\ &\geqslant  D_x w_\varepsilon(v) -c_0-\varepsilon.
\end{align*}
Since $t\mapsto u_\lambda\big(\gamma^x_\lambda(t)\big)$ is Lipschitz continuous, and $\gamma^x_\lambda$ is a $(u_\lambda,L_\lambda,c_0)$-calibrated curve, for a.e. $t<0$ we have
\begin{flalign}\label{eq added}
\frac{d}{dt}u_\lambda\big(\gamma^x_\lambda(t)\big)
&
=L_G\Big(\gamma^x_\lambda(t),\dot{\gamma}^x_\lambda(t),\lambda u_\lambda\big(\gamma^x_\lambda(t)\big)\Big)+c_0
&&
\\
&
\geqslant  L_G\Big(\gamma^x_\lambda(t),\dot{\gamma}^x_\lambda(t),\lambda u_\lambda\big(\gamma^x_\lambda(t)\big)\Big)-L_G\big(\gamma^x_\lambda(t),\dot{\gamma}^x_\lambda(t),0\big)+D_{\gamma^x_\lambda(t)}w_\varepsilon(\dot \gamma^x_\lambda(t))-\varepsilon
&&\nonumber
\\
&=\frac{d}{dt}w_\varepsilon\big(\gamma^x_\lambda(t)\big)+\lambda \frac{\partial L_G}{\partial u}\big(\gamma^x_\lambda(s),\dot{\gamma}^x_\lambda(s),0\big)u_\lambda\big(\gamma^x_\lambda(t)\big)-\varepsilon+\Omega_{\lambda,x}(t),&&\nonumber
\end{flalign}
where
\begin{align*}
\Omega_{\lambda,x}(t):=&L_G\Big(\gamma^x_\lambda(t),\dot{\gamma}^x_\lambda(t),\lambda u_\lambda\big(\gamma^x_\lambda(t)\big)\Big)
\\ &-L_G\big(\gamma^x_\lambda(t),\dot{\gamma}^x_\lambda(t),0\big)-\lambda \frac{\partial L_G}{\partial u}\big(\gamma^x_\lambda(s),\dot{\gamma}^x_\lambda(s),0\big)u_\lambda\big(\gamma^x_\lambda(t)\big).
\end{align*}
{
Let us estimate the error term $\Omega_{\lambda,x}(t)$. Let us set $S:=\{(x,v)\in TM\,:\, \|v\|_x\leqslant \hat \kappa\,\}$, where
$\hat\kappa>0$ is the Lipschitz constant of the curves $\{\gamma^x_\lambda\,:\,\lambda\in (0,\lambda_0)\,\}$, according to
Lemma \ref{gameq}. By (L4) we have
\begin{equation}\label{eq estimate Omega}
|\Omega_{\lambda,x}(t)|
\leqslant
\lambda C \eta_S(\lambda C )
\qquad
\hbox{for all $t\leqslant 0$ and $\lambda\in (0,\lambda_0)$,}
\end{equation}
where $C $ is a uniform bound on the $\|u_\lambda\|_\infty$.
}
By multiplying both sides of \eqref{eq added} by $e^{\lambda\int_t^0\frac{\partial L_G}{\partial u}(\gamma^x_\lambda(s),\dot{\gamma}^x_\lambda(s),0)ds}$ and by rearranging terms, we obtain, for a.e. $t<0$,
\begin{flalign*}
&
\frac{d}{dt}\bigg(u_\lambda\big(\gamma^x_\lambda(t)\big)e^{\lambda\int_t^0\!\frac{\partial L_G}{\partial u}(\gamma^x_\lambda(s),\dot{\gamma}^x_\lambda(s),0)ds}\bigg)
\geqslant
\bigg(\frac{d}{dt}w_\varepsilon\big(\gamma^x_\lambda(t)\big)-\varepsilon+\Omega_{\lambda,x}(t)\bigg)
e^{\lambda\int_t^0\!\frac{\partial L_G}{\partial u}(\gamma^x_\lambda(s),\dot{\gamma}^x_\lambda(s),0)ds}.
&&
\end{flalign*}
Integrating the above inequality over the interval $(-T,0]$  where $T\geqslant  T_0$ as stated in Lemma \ref{>T0}, and using an integration by parts, we have
\begin{flalign*}
&
u_\lambda(x)-u_\lambda\big(\gamma^x_\lambda(-T)\big)e^{\lambda\int_{-T}^0\frac{\partial L_G}{\partial u}(\gamma^x_\lambda(s),\dot{\gamma}^x_\lambda(s),0)ds}
&&
\\
&
\quad\geqslant  w_\varepsilon(x)-w_\varepsilon\big(\gamma^x_\lambda(-T)\big)e^{\lambda\int_{-T}^0\frac{\partial L_G}{\partial u}(\gamma^x_\lambda(s),\dot{\gamma}^x_\lambda(s),0)ds}
&&
\\
&
\quad -\int_{-T}^0w_\varepsilon\big(\gamma^x_\lambda(t)\big)\frac{d}{dt}\bigg(e^{\lambda\int_t^0\frac{\partial L_G}{\partial u}(\gamma^x_\lambda(s),\dot{\gamma}^x_\lambda(s),0)ds}\bigg)dt
+\int_{-T}^0 (\Omega_{\lambda,x}(t)-\varepsilon)e^{\lambda\int_t^0\frac{\partial L_G}{\partial u}(\gamma^x_\lambda(s),\dot{\gamma}^x_\lambda(s),0)ds}.
&&
\end{flalign*}
Letting $\varepsilon\to 0^+$ it follows that,
\begin{flalign*}
u_\lambda(x)&\geqslant  w(x)-(\|w\|_\infty+C) e^{\lambda\int_{-T}^0\frac{\partial L_G}{\partial u}(\gamma^x_\lambda(s),\dot{\gamma}^x_\lambda(s),0)ds}
&&
\\
&-\int_{-T}^0 w\big(\gamma^x_\lambda(t)\big) \frac{d}{dt}\bigg(e^{\lambda\int_t^0\frac{\partial L_G}{\partial u}(\gamma^x_\lambda(s),\dot{\gamma}^x_\lambda(s),0)ds}\bigg)dt+\int_{-T}^0 \Omega_{\lambda,x}(t) e^{\lambda\int_t^0\frac{\partial L_G}{\partial u}(\gamma^x_\lambda(s),\dot{\gamma}^x_\lambda(s),0)ds}.
&&
\end{flalign*}
{  From Lemma \ref{>T0} we infer that the maps
\[
t\mapsto e^{\lambda\int_t^0\frac{\partial L_G}{\partial u}(\gamma^x_\lambda(s),\dot{\gamma}^x_\lambda(s),0)ds}
\qquad
\hbox{and}
\qquad
t\mapsto \frac{d}{dt}\bigg(e^{\lambda\int_t^0\frac{\partial L_G}{\partial u}(\gamma^x_\lambda(s),\dot{\gamma}^x_\lambda(s),0)ds}\bigg)
\]
are in $L^1\big((-\infty,0]\big)$ and converge to $0$ as $t\to -\infty$. By taking also into account \eqref{eq added}, we can send $T\to +\infty$ in the above inequality, to get, by the Dominated Convergence Theorem,
\begin{flalign*}
u_\lambda(x)&\geqslant  w(x)
&&
\\
&-\int_{-\infty}^0w\big(\gamma^x_\lambda(t)\big)\frac{d}{dt}\bigg(e^{\lambda\int_t^0\frac{\partial L_G}{\partial u}(\gamma^x_\lambda(s),\dot{\gamma}^x_\lambda(s),0)ds}\bigg)dt+\int_{-\infty}^0 \Omega_{\lambda,x}(t)e^{\lambda\int_t^0\frac{\partial L_G}{\partial u}(\gamma^x_\lambda(s),\dot{\gamma}^x_\lambda(s),0)ds}
&&
\\
&
=:w(x)-I_\lambda+R_\lambda(x).
&&
\end{flalign*}
}
By Definition \ref{mulam}, we have
\begin{align*}
I_\lambda&=-\lambda\int_{-\infty}^0w\big(\gamma^x_\lambda(t)\big)\frac{\partial L_G}{\partial u}\big(\gamma^x_\lambda(s),\dot{\gamma}^x_\lambda(s),0\big) e^{\lambda\int_t^0\frac{\partial L_G}{\partial u}(\gamma^x_\lambda(s),\dot{\gamma}^x_\lambda(s),0)ds}dt
\\
&
=-\lambda\Big(\int_{-\infty}^0 e^{\lambda \int_t^0 \frac{\partial L_G}{\partial u}(\gamma^x_\lambda(s),\dot{\gamma}^x_\lambda(s),0)ds}dt\Big)
\int_{TM}w(y)\frac{\partial L_G}{\partial u}(y,v,0)\, d\tilde{\mu}^x_\lambda(y,v).
\end{align*}
{According to Lemma \ref{>T0}-(i), for $\lambda\in (0,\lambda_0)$,} we derive that
\begin{flalign*}
\lambda\int_{-\infty}^0 e^{\lambda \int_t^0 \frac{\partial L_G}{\partial u}(\gamma^x_\lambda(s),\dot{\gamma}^x_\lambda(s),0)ds}dt
&
=\frac{\lambda\int_{-\infty}^0 e^{\lambda \int_t^0 \frac{\partial L_G}{\partial u}(\gamma^x_\lambda(s),\dot{\gamma}^x_\lambda(s),0)ds}dt}{\int_{-\infty}^0\frac{d}{dt}\bigg(e^{\lambda \int_t^0 \frac{\partial L_G}{\partial u}(\gamma^x_\lambda(s),\dot{\gamma}^x_\lambda(s),0)ds}\bigg)dt}
&&
\\
&
=-\frac{\int_{-\infty}^0 e^{\lambda \int_t^0 \frac{\partial L_G}{\partial u}(\gamma^x_\lambda(s),\dot{\gamma}^x_\lambda(s),0)ds}dt}{\int_{-\infty}^0 \frac{\partial L_G}{\partial u} (\gamma^x_\lambda({  t}),\dot{\gamma}^x_\lambda({  t}),0)e^{\lambda \int_t^0 \frac{\partial L_G}{\partial u}(\gamma^x_\lambda(s),\dot{\gamma}^x_\lambda(s),0)ds}dt}
&&
\\
&
=-\frac{1}{\int_{TM}\frac{\partial L_G}{\partial u}(y,v,0)d\tilde{\mu}^x_\lambda(y,v)}.
&&
\end{flalign*}
By \eqref{eq added} we also have
\begin{equation*}
|R_\lambda(x)|\leqslant  \lambda C \eta_S(\lambda C )\int_{-\infty}^0 e^{\lambda\int_t^0\frac{\partial L_G}{\partial u}(\gamma^x_\lambda(s),\dot{\gamma}^x_\lambda(s),0)ds}\leqslant  \bigg(\frac{1}{\epsilon_1}+\frac{ e^{\lambda K T_0}-1}{K} \bigg)C \eta_S(\lambda C ).
\end{equation*}
The assertion follows by sending $\lambda\to 0^+$.
\end{proof}

We are now in position to prove the first two main theorems of this section.

\begin{proof}[Proof of Theorem \ref{thm1}] Let $u^*$ be an accumulation point of the $(u_\lambda)_{\lambda \in (0,\lambda_0)}$ as $\lambda\to 0^+$. By Lemma \ref{<u0}, we know that $u^*\in\mathcal S$, so $u^*\leqslant  u_0$ in $M$. Let us prove that $u^*\geqslant u_0$. Fix $x\in M$ and
pick $w\in\mathcal S$. From Lemmas \ref{muisma} and \ref{>u0} we infer that
\[
u^*(x)\geqslant  w(x)-\frac{\int_{TM}w(y)\frac{\partial L_G}{\partial u}(y,v,0)d\tilde{\mu}}{\int_{TM}\frac{\partial L_G}{\partial u}(y,v,0)d\tilde{\mu}},
\]
for some Mather measure $\tilde\mu$. Since $w$ satisfies the constraint (\ref{subcon}), we get from this that $u^*(x)\geqslant  w(x)$, hence
$u^*(x)\geqslant \sup\limits_{w\in\mathcal S} w(x)=:u_0(x)$. By the arbitrariness of the choice of $x\in M$, we get that $u_0$ is finite-valued and that $u^*\geqslant  u_0$ in $M$. We conclude that $u_0$ is the unique accumulation point of the family of functions $(u_\lambda)_{\lambda\in (0,\lambda_0)}$ as $\lambda\to 0^+$. The proof is complete.
\end{proof}

\begin{proof}[Proof of Theorem \ref{thm2}]
We denote
\[
\hat u_0(x):=\inf_{\tilde{\mu}\in\widetilde{\mathfrak M}}\frac{\int_{TM} h(y,x)\frac{\partial L_G}{\partial u}(y,v,0)d\tilde\mu(y,v)}{\int_{TM} \frac{\partial L_G}{\partial u}(y,v,0)d\tilde\mu(y,v)},\qquad x\in M.
\]
Note that $\hat u_0$ is finite-valued, as $\hat u_0\geq \min\limits_{M\times M} h>-\infty$. We start by remarking that $\hat u_0$ is a subsolution of \eqref{E0}. Indeed, for every fixed $\tilde\mu\in \Mtilde$, the function
$h_{\tilde\mu}:M\to\R, \ x\mapsto  \int_{TM} h(y,x)\,  d  {\tilde\mu} (y)$ is a convex combination of the family of
critical solutions $(h_y)_{y\in M}$, where $h_y(x)=h(y,x)$. By the convexity of $H$
in the momentum and the equi-Lipschitz character of the critical subsolutions, see Propositions \ref{prop when G convex} and \ref{prop equivalence}, it follows that each  $h_{\tilde\mu}$ is a critical subsolution. By Proposition  \ref{prop when G convex} again, we infer that
a finite valued infimum of critical subsolutions is itself
a critical subsolution. Therefore $\hat u_0$ is a critical subsolution.

Let us now prove that $u_0\leqslant  \hat u_0$. By Proposition \ref{prop equivalence} we know that
$u_0(x)\leqslant  u_0(y)+h(y,x)$ for all $x,y\in M$. Let us integrate this inequality with respect to a Mather measure
$\tilde{\mu}\in\widetilde{\mathfrak M}$.  By assumption and by definition of Mather set, we have $\frac{\partial L_G}{\partial u}(x,v,0)\leqslant  0$ for all $(x,v)\in\textrm{supp}(\tilde\mu)$. We infer
\begin{align*}
&u_0(x)\int_{TM}\frac{\partial L_G}{\partial u}(y,v,0)\,d\tilde\mu(y,v)
\\ &\geqslant  \int_{TM}u_0(y)\frac{\partial L_G}{\partial u}(y,v,0)\,d\tilde\mu(y,v)+\int_{TM}h(y,x)\frac{\partial L_G}{\partial u}(y,v,0)\,d\tilde\mu(y,v)
\\ &\geqslant  \int_{TM}h(y,x)\frac{\partial L_G}{\partial u}(y,v,0)\,d\tilde\mu(y,v),
\end{align*}
where, for the last inequality, we used that $u_0$ satisfies \eqref{subcon}. From this we get
\[
u_0(x)
\leqslant
\dfrac{\int_{TM}h(y,x)\frac{\partial L_G}{\partial u}(y,v,0)\,d\tilde\mu(y,v)}{\int_{TM}\frac{\partial L_G}{\partial u}(y,v,0)\,d\tilde\mu(y,v)}
\qquad
\hbox{for all $x\in M$.}
\]
By taking the inf with respect to $\tilde\mu\in{\widetilde{\mathcal M}}$ of the right-hand side term in the above inequality, we get $u_0\leqslant  \hat u_0$.

Last, we prove that $u_0\geqslant  \hat u_0$. For every fixed $z\in \A$, set
\[
U_{z}(x):=-h(x,z)+\hat u_0(z), \qquad x\in M.
\]
Then $U_z$ is a subsolution of (\ref{E0}), by Proposition \ref{prop h}. Furthermore, for every $x\in M$ and
$\tilde\mu\in\widetilde{\mathfrak M}$, we have
\[
\frac{\int_{TM} U_z(x)\,\frac{\partial L_G}{\partial u}(x,v,0)d\tilde\mu(x,v)}{\int_{TM} \frac{\partial L_G}{\partial u}(x,v,0)d\tilde\mu(x,v)}
=
-\frac{\int_{TM} h(x,z)\,\frac{\partial L_G}{\partial u}(x,v,0)d\tilde\mu(x,v)}{\int_{TM} \frac{\partial L_G}{\partial u}(x,v,0)d\tilde\mu(x,v)}
+
\hat u_0(z)
\leqslant  0,
\]
which implies that $U_z\in\mathcal S$. Then $u_0(x)\geqslant   U_z(x)$ for all $x\in M$. In particular, by taking $x=z\in\mathcal A$, we get
\[
u_0(z)\geqslant  U_z(z)=-h(z,z)+\hat u_0(z)=\hat u_0(z).
\]
We have thus shown that $u_0\geqslant \hat u_0$ on $\A$, hence $u_0\geqslant \hat u_0$ on $M$ according according to Proposition \ref{prop AC}.
\end{proof}

We proceed by proving the third main theorem of this section, namely the trichotomy result stated in Theorem \ref{thm Maxime bis}. We start by establishing a sort of Harnack-type inequality for subsolutions of \eqref{E}.

\begin{lemma}
There exists  a constant $A_+>0$ such that, if $\lambda \in (0,1)$ and $w_\lambda : M\to \R$ is a subsolution to \eqref{E}, then
\begin{equation}\label{Harnack}
\min_{x\in M} w_\lambda(x) \leqslant \max_{x\in M} w_\lambda(x) \leqslant \min_{x\in M} w_\lambda(x) +A_+\Big(1+\lambda\min_{x\in M}
|w_\lambda(x)|\Big).
\end{equation}
\end{lemma}

\begin{proof}
The left hand side inequality is obvious.
Let us prove the right hand side inequality.
Let $x_\lambda\in M$ such that $M_\lambda = w_\lambda(x_\lambda)= \max w_\lambda$.
Let $y_\lambda\in M$ such that $m_\lambda = w_\lambda(y_\lambda)= \min w_\lambda$.
 Denote $d_\lambda:=d(y_\lambda,x_\lambda)$. Let $\zeta:[0,d_\lambda]\rightarrow M$ be a geodesic satisfying $\zeta(0)=y_\lambda$ and $\zeta(d_\lambda)=x_\lambda$ with constant speed 1.
  By $w_\lambda \prec L_\lambda+c_0$ and (L1), we get
\begin{flalign*}
  w_\lambda\big(\zeta(s)\big)
  &\leqslant w_\lambda(y_\lambda)
  +\int_0^s \Big[{  L_G}\Big(\zeta(\tau),\dot \zeta(\tau),\lambda w_\lambda\big(\zeta(\tau)\big)\Big)+c_0\Big]d\tau
  &&
  \\ &
  \leqslant m_\lambda
  +\int_0^s \Big[{  L_G}\big(\zeta(\tau),\dot \zeta(\tau),\lambda m_\lambda \big)+c_0+\lambda K\big(w_\lambda\big(\zeta(\tau)\big)-m_\lambda\big)\Big]d\tau
  &&
  \\
  &\leqslant
  m_\lambda+\int_0^s \Big[{  L_G}\big(\zeta(\tau),\dot \zeta(\tau),0 \big)+c_0+\lambda K |m_\lambda|\Big] d\tau+\lambda K\int_0^s \Big[w_\lambda\big(\zeta(\tau)\big)-m_\lambda\Big]d\tau
  &&
\\
  &\leqslant m_\lambda+(C_{  L_G}+\lambda K |m_\lambda|)d_\lambda+\lambda K\int_0^s \Big[w_\lambda\big(\zeta(\tau)\big)-m_\lambda\Big]d\tau
  &&
\end{flalign*}
where $C_{  L_G}:=\max\limits_{x\in M,\|v\|_x\leqslant 1}|{  L_G}(x,v,0)+c_0|$. By the Gronwall inequality we infer
\begin{equation*}
  w_\lambda\big(\zeta(s)\big)-m_\lambda \leqslant (C_{  L_G}+\lambda K|m_\lambda|)d_\lambda e^{\lambda Ks}\leqslant (C_{  L_G}+\lambda K|m_\lambda|)d_\lambda e^{\lambda K d_\lambda},\quad \forall s\in(0,d_\lambda].
\end{equation*}
Taking $s=d_\lambda$, and recalling that $\lambda \in (0,1)$,we have
\[M_\lambda = w_\lambda(x_\lambda)\leqslant m_\lambda+(C_{  L_G}+\lambda K |m_\lambda|)d_\lambda e^{\lambda K d_\lambda}\leqslant m_\lambda+(C_{  L_G}+\lambda K|m_\lambda|)\overline De^{ K\overline D},\]
where $\overline D:=$diam$(M)$. The result follows taking $A_+ = \max (C_{  L_G} , K)\overline De^{ K\overline D}$.
\end{proof}

As a consequence, we derive the following key proposition.
It will be also used  in Section \ref{sec model case} to show the existence of diverging families of solutions.

\begin{proposition}\label{uto-inf}
Let $\Lambda$ be a subset of $(0,1)$ having $0$ as accumulation point.
Let $(w_\lambda)_{\lambda \in\Lambda}$ be a family of  subsolutions of (\ref{E}).
\begin{itemize}
\item[\em (i)]
 If, for each $\lambda\in \Lambda$, there is a point $x_\lambda\in M$ such that $w_\lambda(x_\lambda)\to-\infty$ as $\lambda\to 0^+$, $\lambda \in \Lambda$, then $w_\lambda$ uniformly converges to $-\infty$ as $\lambda\to 0^+$, $\lambda \in \Lambda$.
 \item[\em (ii)]
 If, for each $\lambda\in \Lambda$, there is a point $x_\lambda\in M$ such that $w_\lambda(x_\lambda)\to+\infty$ as $\lambda\to 0^+$, $\lambda\in\Lambda$, then $w_\lambda$ uniformly converges to $+\infty$ as $\lambda\to 0^+$, $\lambda \in \Lambda$.
 \end{itemize}
\end{proposition}

\begin{proof}
As in the previous proof, we denote $M_\lambda =  \max w_\lambda$ and $m_\lambda = \min w_\lambda$.

 Let us prove assertion (i). The hypothesis $w_\lambda(x_\lambda)\to-\infty$ is equivalent to $m_\lambda \to -\infty$ as $\lambda\to 0^+$, $\lambda \in \Lambda$. Then restricting to $\lambda \in \Lambda \cap (0,1)$, \eqref{Harnack} implies that   $m_\lambda \sim M_\lambda$ hence $M_\lambda \to -\infty$ as $\lambda\to 0^+$, $\lambda \in \Lambda$.

 Let us prove item (ii). The hypothesis $w_\lambda(x_\lambda)\to+\infty$ is equivalent to $M_\lambda \to +\infty$ as $\lambda\to 0^+$, $\lambda \in \Lambda$. Then restricting to $\lambda \in \Lambda \cap (0,1)$, \eqref{Harnack} implies that  $m_\lambda \sim M_\lambda$ hence $m_\lambda \to +\infty$ as $\lambda\to 0^+$, $\lambda \in \Lambda$.
\end{proof}

We are now in position to prove Theorem \ref{thm Maxime bis}. We recall that, in what follows, $u_0$ still denotes the critical solution provided by Theorem \ref{thm1}

\begin{proof}[Proof of Theorem \ref{thm Maxime bis}]
Let us denote by $S_\lambda$ the set of solutions to \eqref{E}. Define
\[S_\lambda^-:=\{v_\lambda\ \textrm{is\ a\ solution\ of\ (\ref{E})\ with}\ v_\lambda<u_0-1\}\]
and
\[\psi(\lambda):=\sup \{v_\lambda(x):\ v_\lambda\in S_\lambda^-,\ x\in M\}.\]
If $S_\lambda^-=\emptyset$, we set $\psi(\lambda)=-\infty$.

We claim that $\lim\limits_{\lambda\to 0}\psi(\lambda)=-\infty$. Otherwise, there is $A>0$, sequences $\lambda_n\to 0^+$, $v_{\lambda_n}\in S_{\lambda_n}$ and $y_n\in M$ such that
\begin{equation}\label{>-A}
  v_{\lambda_n}(y_n)>-A,\quad \forall n\in\N.
\end{equation}
We claim that $v_{\lambda_n}$ is uniformly bounded from below. If not, there is $z_n\in M$ such that
\[
v_{\lambda_n}(z_n)\to-\infty\qquad \hbox{as $n\to +\infty$}.
\]
By Proposition \ref{uto-inf}, $v_{\lambda_n}$ uniformly converges to $-\infty$, which contradicts (\ref{>-A}).
Then, according to Theorem \ref{thm1} and Remark \ref{oss picky}, the sequence $v_{\lambda_n}$ uniformly converges. to the only possible limit $u_0$, which contradicts the fact that $v_{\lambda_n}\in S_{\lambda_n}$ for all $n\in\N$.

In a similar manner, define
\[S_\lambda^+:=\{v_\lambda\ \textrm{is\ a\ solution\ of\ (\ref{E})\ with}\ v_\lambda>u_0+1\}\]
and
\[\varphi(\lambda):=\inf \{v_\lambda(x):\ v_\lambda\in S_\lambda^+,\ x\in M\}.\]
If $S_\lambda^+=\emptyset$, we set $\varphi(\lambda)=+\infty$. The same proof yields that $\lim\limits_{\lambda\to 0}\varphi(\lambda)=+\infty$.

Define
\[\theta(\lambda):=\sup\{|v_\lambda(x)-u_0(x)|,\ v_\lambda\in S_\lambda \setminus (S_\lambda^- \cup S_\lambda^+),\ x\in M\}.\]
We claim that $\lim\limits_{\lambda\to 0}\theta(\lambda)=0$. Otherwise, there is $\varepsilon>0$, a sequence of discount factors $\lambda_n\to 0^+$ and of solutions $v_{\lambda_n}\in S_{\lambda_n}\setminus( S^-_{\lambda_n}\cup S^+_{\lambda_n})$ and $y_n\in M$ such that
\begin{equation}\label{petit}
|v_{\lambda_n}(y_n)-u_0(y_n)|>\varepsilon,\quad \forall n.
\end{equation}
Since $v_{\lambda_n}\in S_{\lambda_n}\setminus( S^-_{\lambda_n}\cup S^+_{\lambda_n})$, there is also a point $x_n$ such that
\[|v_{\lambda_n}(x_n) -  u_0(x_n)|\leqslant 1.\]
Similarly to the first step, we infer from Proposition \ref{uto-inf} that $v_{\lambda_n}$ is uniformly bounded. Then, according to Theorem \ref{thm1} and Remark \ref{oss picky}, the sequence $v_{\lambda_n}$ uniformly converges to the only possible limit $u_0$. But again this yields a contradiction as $\|v_{\lambda_n}-u_0\|_\infty >\varepsilon$ for all $n\in\N$.
This concludes the proof.
\end{proof}

\section{The linear case}\label{sec model case}

{In this section we continue our analysis on the vanishing discount problem in the case when $G$ depends linearly on $u$. Hence we will}
consider a Hamilton-Jacobi equation with discount factor $\lambda>0$ of the form
\begin{equation}\label{VH}\tag{${\textrm{HJ}}_\lambda$}
  \lambda a(x)u(x)+H(x,D_x u)=c_0\qquad\hbox{in $M$},
\end{equation}
where we assume that $H:T^*M\to\mathbb R$ is a continuous function satisfying (H1)-(H2) (convexity and superlinearity) and the coefficient $a$ is a continuous function on $M$ satisfying the following condition:
\begin{equation}\label{V>0}\tag{$a1$}
  \int_{TM}a(x)d\tilde{\mu}>0\qquad \hbox{for all $\tilde \mu\in\widetilde{\mathfrak M}$},
\end{equation}
where $\Mis$ denotes the compact and convex subset of $\parts(TM)$ made up by Mather measures associated with $H$.
Without any loss of generality, we shall restrict to the case $\lambda\in(0,1)$. Equations of the form \eqref{VH} can be regarded
as a model example for the more general Hamilton-Jacobi equations of contact type considered in Section \ref{sec general}, cf. equation \eqref{E}. The full Hamiltonian is then given by $G(x,p,u) = a(x)u+H(x,p)$. The Hamiltonian $G$ then fulfills all the hypotheses of the previous section. In particular, it is Lipschitz with respect to $u$ with Lipschitz constant $\|a\|_\infty$.

We will be concerned with the issue of existence of solutions to equation \eqref{VH}, at least for small values of $\lambda\in (0,1)$, and their  asymptotic behavior as $\lambda\to 0^+$. In Section \ref{sec model converging} we show the existence of the maximal solution to
\eqref{VH} for values of $\lambda\in (0,1)$ small enough. These solutions are shown to be equi-bounded and equi-Lispchitz, therefore, in view of the results established in Section \ref{sec general}, they converge to a solution of the limit critical equation
\begin{equation}\label{e0}\tag{HJ$_0$}
  H(x,Du)=c_0\quad \hbox{in $M$.}
\end{equation}
In Section \ref{sec model diverging}, under further natural conditions on the coefficient $a$, we investigate the existence of possible other families of solutions of \eqref{VH} and we describe their asymptotic behavior as $\lambda\to 0^+$.
\subsection{A converging family of solutions}\label{sec model converging}


We recall that  $L:TM\to\mathbb R$ is the Lagrangian associated with $H$ via the Legendre transform, namely
\[
L(x,v):=\sup_{p\in T^*_x M}\big(p(v)-H(x,p)\big)\qquad\hbox{for all $(x,v)\in TM$}.
\]
The constant $c_0\in \R$ is the critical value of $H$. We define the following value function, for $x\in M$,
\[V_\lambda(x):=\inf_{\gamma}\limsup_{t\to+\infty}\int_{-t}^0e^{-\lambda \int_s^0a(\gamma(\tau))d\tau}\Big(L\big(\gamma(s),\dot\gamma(s)\big)+c_0\Big)ds,\]
where the infimum is taken among absolutely continuous curves $\gamma:(-\infty,0]\to M$ with $\gamma(0)=x$.
{This definition is inspired by the formula given in \cite[Theorem 4.8]{Z}  to represent the
unique solution of \eqref{VH} in the case $a\geqslant 0$ in $M$ and $a>0$ on $\A$.}

The following holds.

\begin{theorem}\label{thm3}
There exists $\lambda_0\in (0,1)$ such that for every $\lambda\in(0,\lambda_0)$ the following holds:
\begin{itemize}
\item[\em (i)]  the value function $V_\lambda : M\to \R$ is finite-valued;\smallskip
\item[\em (ii)] the functions $\{V_\lambda\,:\,\lambda\in (0,\lambda_0)\,\}$ are equi-bounded and equi-Lipschitz;\smallskip
\item[\em (iii)]  the value function $V_\lambda$ is the maximal subsolution of (\ref{VH}). In particular, it is the maximal solution of (\ref{VH});
\item[\em (iv)] for every $x\in M$, there exists a curve $\gamma^x_\lambda:(-\infty,0]\to M$ such that
\[
V_\lambda(x)=\int_{-\infty}^0e^{-\lambda \int_s^0a(\gamma^x_\lambda(\tau))d\tau}\Big(L\big(\gamma^x_\lambda(s),\dot\gamma^x_\lambda(s)\big)+c_0\Big)ds.
\]
Furthermore, the curve $\gamma^x_\lambda$ is $\hat \kappa$-Lipschitz, for some $\hat \kappa>0$ independent of $\lambda\in (0,\lambda_0)$ and $x\in M$.
\end{itemize}
\end{theorem}

\begin{remark}
For every $\lambda>0$, define
\[c_\lambda:=\inf_{u\in C^\infty(M)}\sup_{x\in M}\big\{H(x,Du)+\lambda a(x)u\big\}.\]
According to \cite{JYZ,n2}, if (\ref{VH}) has solutions, then $c_0\geqslant  c_\lambda$.
Theorem \ref{thm3} implies that $c_0\geqslant  c_\lambda$ for every $\lambda>0$ small enough if (\ref{V>0}) holds. We also point out that the value provided by the above inf-sup with $\lambda=0$ agrees with the critical constant $c_0$, as it is well known, see for instance \cite{FS,CIP}.
\end{remark}

According to Theorem \ref{thm1}, the functions $V_\lambda$ uniformly converge to a critical solution $u_0$ as $\lambda\to 0^+$, where $u_0$ is the maximal subsolution $w$ of (\ref{e0}) satisfying
\begin{equation*}
  \int_{TM}w(x)a(x)d\tilde{\mu}(x,v)\leqslant  0,\quad \forall \tilde{\mu}\in\widetilde{\mathfrak M}.
\end{equation*}

We can furthermore strengthen the conclusion of Theorem \ref{thm Maxime bis} on the asymptotic behavior, as
$\lambda\to 0^+$, of all possible families of solutions to equation  \eqref{VH}. Theorem \ref{thm3} in fact rules out the
possibility that there exist families of solutions that diverge to $+\infty$. The precise statement is the following.

\begin{theorem}\label{thm Maxime}
There exist $\psi:(0,1)\to\mathbb [-\infty,+\infty)$ and $\theta:(0,1)\to\mathbb R$ with
\[\lim_{\lambda\to 0}\psi(\lambda)=-\infty,\quad \lim_{\lambda\to 0}\theta(\lambda)=0\]
such that, if $v_\lambda$ is a solution of (\ref{VH}), then either $v_\lambda\leqslant  \psi(\lambda)$ or $\|v_\lambda-u_0\|_\infty\leqslant  \theta(\lambda)$ for all $\lambda\in (0,1)$.
\end{theorem}

The remainder of this subsection is devoted to prove the statement of Theorem \ref{thm3}. This will be obtained via a series of intermediate results.

A key step in order to establish that the value function is a (Lispchitz) viscosity subsolution to \eqref{e0} is to prove that it satisfies the Dynamic Programming Principle. Please note that the next proposition is valid even when $V_\lambda$ is not finite-valued.

\begin{proposition}\label{prop DPP}
(Dynamic Programming Principle). Let $\lambda\in (0,1)$. For each absolutely continuous curve $\gamma:[{b_1},b_2]\to M$, we have
\begin{eqnarray*}
e^{-\lambda \int_{b_2}^0 a(\gamma(\tau))d\tau} V_\lambda\big(\gamma(b_2)\big)\leqslant
e^{-\lambda \int_{b_1}^0 a(\gamma(\tau))d\tau} V_\lambda\big(\gamma({b_1})\big)+ \int_{{b_1}}^{b_2}e^{-\lambda \int_s^0a(\gamma(\tau))d\tau}\Big(L\big(\gamma(s),\dot\gamma(s)\big)+c_0\Big)ds.
\end{eqnarray*}
\end{proposition}

\begin{proof}
We start by claiming that we can reduce to the case ${b_2}=0$, without any loss of generality. Indeed, by multiplying the above inequality for $e^{\lambda \int_{b_2}^0 a(\gamma(\tau))d\tau}$, we get
\begin{eqnarray*}
V_\lambda\big(\gamma({b_2})\big)
\leqslant
e^{-\lambda \int_{b_1}^{b_2} a(\gamma(\tau))d\tau} V_\lambda\big(\gamma({b_1})\big)
+
\int_{{b_1}}^{b_2}e^{-\lambda \int_s^{b_2} a(\gamma(\tau))d\tau}\Big(L\big(\gamma(s),\dot\gamma(s)\big)+c_0\Big)ds.
\end{eqnarray*}
The claim follows by making the change of variables $\tau':=\tau-{b_2}$, $s':=s-{b_2}$ and by replacing $\gamma$ with $\gamma_{-{b_2}}(\cdot):=\gamma(\cdot+{b_2})$.

Let us then prove the assertion with ${b_2}=0$. For each $\xi:(-\infty,0]\to M$ with $\xi(0)=\gamma({b_1})$, we define
\begin{equation*}
  \bar \gamma(t):=
   \begin{cases}
   \gamma(t) &\hbox{if $t\in [{b_1},0]$}\\
   \xi(t-{b_1})   &\hbox{if $t\in (-\infty,{b_1})$}
   \end{cases}
   \end{equation*}
Then, by definition of the value function, we have
\begin{align*}
V_\lambda\big(\gamma(0)\big)&\leqslant  \limsup_{t\to+\infty}\int_{-t}^0e^{-\lambda \int_s^0{a}(\bar\gamma(\tau))d\tau}\Big(L\big(\bar\gamma(s),\dot{\bar \gamma}(s)\big)+c_0\Big)ds
\\ &=\int_{{b_1}}^0e^{-\lambda\int_s^0a(\gamma(\tau))d\tau}\Big(L\big(\gamma(s),\dot{\gamma}(s)\big)+c_0\Big)ds
\\ &\quad +e^{-\lambda \int_{b_1}^0a(\gamma(\tau))d\tau}\limsup_{t\to+\infty}\int_{-t}^0e^{-\lambda \int_s^0a(\xi(\tau))d\tau}\Big(L\big(\xi(s),\dot{\xi}(s)\big)+c_0\Big)ds.
\end{align*}
Taking the infimum among all $\xi$, we get the assertion.
\end{proof}

We now proceed to show $V_\lambda$ is  a bounded function on $M$, at least for $\lambda\in (0,1)$ small enough. We start by proving the following upper bound.

\begin{lemma}\label{lemma bdd above}
There is a constant $\widehat C_v^+>0$ such that $V_\lambda(x)\leqslant  \widehat C_v^+/\lambda$ for all $x\in M$ and $\lambda\in (0,1)$.
\end{lemma}
\begin{proof}
By condition (\ref{V>0}), there is a point $x_0\in M$ such that $a(x_0)>0$. We denote $d:=d(x_0,x)$, $\overline{D}:=\textrm{diam}(M)$. We take a geodesic $\zeta:[-d,0]\to M$ with $\zeta(-d)=x_0$ and $\zeta(0)=x$. Define
\begin{equation*}
  \bar \gamma(t)=
   \begin{cases}
   \zeta(t), & -d\leqslant  t\leqslant  0,
   \\
   x_0, & t<-d
   \\
   \end{cases}
\end{equation*}
and set $C_L:=\max\limits_{x\in M,\|v\|_x\leqslant  1}|L(x,v)+c_0|$.
Then
\begin{align*}
V_\lambda(x)&\leqslant  \limsup_{t\to+\infty}\int_{-t}^0e^{-\lambda \int_s^0a(\bar \gamma(\tau))d\tau}\Big(L\big(\bar \gamma(s),\dot{\bar \gamma}(s)\big)+c_0\Big)ds
\\ &=\int_{-d}^0e^{-\lambda \int_s^0a(\zeta(\tau))d\tau}\Big(L\big(\zeta(s),\dot\zeta(s)\big)+c_0\Big)ds
\\ &\quad +e^{-\lambda \int_{-d}^0a(\zeta(\tau))d\tau}\limsup_{t\to+\infty}\int_{-t}^{-d}e^{-\lambda \int_s^{-d}a(x_0)d\tau}\big(L(x_0,0)+c_0\big)ds\\
&\leqslant  C_L\left( \int_{-d}^0e^{-\lambda\|a\|_\infty s}ds+e^{\lambda \|a\|_\infty d}\int_{-\infty}^{-d}e^{\lambda a(x_0)(d+s)}ds\right)\\
&=C_L\left( \frac{e^{\lambda\|a\|_\infty  d}-1}{\lambda\|a\|_\infty}+\frac{e^{\lambda \|a\|_\infty  d}}{\lambda a(x_0)}\right)
\leqslant  \frac{2 C_L\, e^{\|a\|_\infty\overline{D}}}{\lambda a(x_0)}.\qedhere
\end{align*}
\end{proof}
\ \medskip\\
Now we know that the infimum in $V_\lambda(x)$ is taken among curves in
\begin{flalign*}
&
\quad
\Gamma_\lambda:=\bigg\{\gamma:(-\infty,0]\to M:\ \limsup_{t\to+\infty}\int_{-t}^0e^{-\lambda \int_s^0a(\gamma(\tau))d\tau}\Big(L\big(\gamma(s),\dot\gamma(s)\big)+c_0\Big)ds\leqslant  \frac{\widehat C_v^+}{\lambda}\bigg\}.
&&
\end{flalign*}
For $\gamma:(-\infty,0]\to M$ and $t>0$, we define
\[
\alpha(\lambda,\gamma,t):=\int_{-t}^0e^{-\lambda \int_s^0a(\gamma(\tau))d\tau}ds
\]
and
\[\beta(\lambda,\gamma,t):=e^{-\lambda \int_{-t}^0a (\gamma(s))ds}.\]
Since $\frac{d}{dt}\alpha(\lambda,\gamma,t)=\beta(\lambda,\gamma,t)>0$, $t\mapsto \alpha(\lambda,\gamma,t)$ is increasing. Then $\lim\limits_{t\to+\infty}\alpha(\lambda,\gamma,t)$ exists for each $\lambda$ and $\gamma$, and may equal $+\infty$. We will also need the following auxiliary remark.

\begin{lemma}\label{a/b}
For every $\lambda>0$ and for all curves $\gamma$, we have
\[-\lambda \|a\|_\infty+\frac{1}{\alpha(\lambda,\gamma,t)}\leqslant  \frac{\beta(\lambda,\gamma,t)}{\alpha(\lambda,\gamma,t)}\leqslant  \lambda \|a\|_\infty+\frac{1}{\alpha(\lambda,\gamma,t)},\quad \forall t>0.\]
\end{lemma}
\begin{proof}
A direct calculation shows that
\[\frac{\frac{d}{dt}\beta(\lambda,\gamma,t)}{\frac{d}{dt}\alpha(\lambda,\gamma,t)}=-\lambda a\big(\gamma(-t)\big).\]
By the Cauchy mean value theorem, there is $\xi\in(0,t)$ such that
\[\frac{\beta(\lambda,\gamma,t)-\beta(\lambda,\gamma,0)}{\alpha(\lambda,\gamma,t)-\alpha(\lambda,\gamma,0)}=\frac{\beta(\lambda,\gamma,t)-1}{\alpha(\lambda,\gamma,t)}=\frac{\frac{d}{dt}\beta(\lambda,\gamma,\xi)}{\frac{d}{dt}\alpha(\lambda,\gamma,\xi)}=-\lambda a\big(\gamma (-\xi)\big),\]
which implies the conclusion. 
\end{proof}

We distill in the next lemma an argument that we will repeatedly use in the sequel.

\begin{lemma}\label{lemma Cn}
Assume there are sequences $\lambda_n\to 0^+$, $\gamma_n:(-\infty,0]\to M$, $t_n\in (0,+\infty)$ and $\theta_n\to 0^+$ such that
\begin{flalign}
 &\qquad \alpha(\lambda_n,\gamma_n,t_n):=\int_{-t_n}^0e^{-\lambda_n \int_s^0a(\gamma_n(\tau))d\tau}ds\to +\infty, \qquad\hbox{as $n\to +\infty$,}&& \label{eq diverging alpha}\\
  &\qquad \frac{1}{\alpha(\lambda_n,\gamma_n,t_n)}\int_{-t_n}^0e^{-\lambda_n\int_s^0a(\gamma_n(\tau))d\tau}\Big(L\big(\gamma_n(s),\dot{\gamma}_n(s)\big)+c_0\Big)ds\leqslant  \theta_n,
  \quad\forall n\in\N.&&
  \label{leqCn}
\end{flalign}
Define a probability measure $\tilde{\mu}_n\in\parts (TM)$ by
\begin{equation}\label{mun}
\int_{TM}f\, d\tilde{\mu}_n:=\frac{\int_{-t_n}^0e^{-\lambda_n\int_s^0a(\gamma_n(\tau))d\tau}f\big(\gamma_n(s),\dot{\gamma}_n(s)\big)ds}{\alpha(\lambda_n,\gamma_n,t_n)},\qquad \forall f\in C_\ell(TM).\tag{$\star$}
\end{equation}
Then the set $(\tilde\mu_n)_n$ is relatively compact in $\mathscr P_\ell$ (for the weak-$*$ topology coming from $C_\ell(TM)$) and any of its accumulation points is a Mather measure associated with $L$.
\end{lemma}

\begin{proof}
Note that, due to \eqref{eq diverging alpha}, we have $t_n\to +\infty$ as $n\to +\infty$.
Since $L$ is uniformly superlinear in the fibers, there exists a constant $C_1\geqslant   0$ such that
\[
\max_{n\in\N}\, \theta_n
\geqslant
\int_{TM}\big(L(x,v)+c_0\big)\,d\tilde{\mu}_n(x,v)\geqslant  \int_{TM}(\|v\|_x-C_1)d\tilde{\mu}_n
\qquad\hbox{for all $n\in\N$,}
\]
which readily implies that the sequence $(\tilde{\mu}_n)_n$ is a well defined sequence in $\mathscr P_\ell$. The asserted precompactness of
$(\tilde{\mu}_n)_n$ in $\parts(TM)$ follows from Lemma \ref{technical}.

Let $\tilde{\mu}$ be a limit point of a subsequence of $(\tilde{\mu}_n)_n$, that we will not relabel to ease notations.
We are going to show that $\tilde\mu$ is a Mather measure.\smallskip

\noindent {{\bf The measure $\tilde\mu$ is closed}:} we pick $\phi\in C^1(M)$. An integration by parts shows that
\begin{align}
\nonumber&\int_{TM} D_x\phi(v)\,d\tilde{\mu}_n(x,v)\\
&=
\frac{\int_{-t_n}^0 \frac{d\phi}{dt}\big(\gamma_n(s)\big)\,  e^{-\lambda_n\int_s^0a(\gamma_n(\tau))d\tau}\,ds   }{\alpha(\lambda_n,\gamma_n,t_n)}
\label{IPP}\\
\nonumber&=\frac{\phi\big(\gamma_n(0)\big)-e^{-\lambda_n\int_{-t_n}^0a(\gamma_n(\tau))d\tau}\phi\big(\gamma_n(-t_n)\big)}{\alpha(\lambda_n,\gamma_n,t_n)}-\lambda_n\int_{TM}a(x)\phi(x)\,d\tilde{\mu}_n(x,v).
\end{align}
We infer
\begin{align*}
\bigg|\int_{TM} D_x\phi(v) d\tilde{\mu}_n\bigg|
&\leqslant
\frac{\|\phi\|_\infty}{\alpha(\lambda_n,\gamma_n,t_n)}+\frac{e^{-\lambda_n\int_{-t_n}^0a(\gamma_n(\tau))d\tau}}{\alpha(\lambda_n,\gamma_n,t_n)}\|\phi\|_\infty+\lambda_n\|a\|_\infty\|\phi\|_\infty \\
&=\bigg(\frac{1}{\alpha(\lambda_n,\gamma_n,t_n)}+\frac{\beta(\lambda_n,\gamma_n,t_n)}{\alpha(\lambda_n,\gamma_n,t_n)}+\lambda_n\|a\|_\infty\bigg)\|\phi\|_\infty.
\end{align*}
By sending $n\to +\infty$ and by using Lemma \ref{a/b}, hypothesis \eqref{eq diverging alpha} and again Lemma \ref{technical},
we derive
\[
\int_{TM} D_x\phi(v)\, d\tilde{\mu}(x,v)=\lim_{n\to+\infty}\int_{TM} D_x\phi(v)\, d\tilde{\mu}_n(x,v)=0.\]

\noindent {\bf {The measure $\tilde{\mu}$ is minimizing}:} by (\ref{leqCn}) we get
\[
0
=
\lim_{n\to+\infty}\theta_n
\geqslant
\liminf_{n\to+\infty}\int_{TM}\big(L(x,v)+c_0\big)\, d\tilde\mu_n(x,v)
\geqslant
\int_{TM}\big(L(x,v)+c_0\big)\, d\tilde{\mu}(x,v),
\]
where the last inequality follows from the fact that $L$ is continuous and bounded from below and the measures
$\tilde\mu_n$ are weakly-$*$ converging to $\tilde\mu$ in $\PP_\ell$ (hence, narrowly converging), see for instance \cite[Section 5.1.1]{Gigli_book}.

In view of Theorem \ref{Mather},
we conclude that $\tilde\mu$ is a Mather measure.
\end{proof}

\noindent Let $C>0$ be a fixed constant  and define, for any fixed integer $p\geqslant  0$,
\begin{eqnarray*}
\quad \Gamma_\lambda(p):=\bigg\{\gamma:(-\infty,0]\to M:\ \limsup_{t\to+\infty}\int_{-t}^0e^{-\lambda \int_s^0a(\gamma(\tau))d\tau}\big(L\big(\gamma(s),\dot\gamma(s)\big)+c_0\big)ds\leqslant  \frac{C}{\lambda^p}\bigg\}.
\end{eqnarray*}

The information provided by the next lemma will be crucial for our upcoming analysis.

\begin{lemma}\label{Ca}
Let $p\geqslant  0$ be a fixed integer and assume that $\Gamma_\lambda(p)\not=\emptyset$ for every $\lambda\in (0,1)$ small enough. Then there exists $\lambda(p)\in (0,1)$ such that
\begin{eqnarray*}
A(p):=\sup\Big\{\lambda^{p+1} \alpha(\lambda,\gamma,t)\,:\, \lambda\in\big(0,\lambda(p)\big),\ \gamma\in\Gamma_\lambda(p),\ t>0\Big\}<+\infty.
\end{eqnarray*}

\end{lemma}

\begin{remark}\label{rmk Ca}
Note that $\Gamma_\lambda(p+1)\supseteq \Gamma_\lambda(p)$ for every integer $p\geqslant  0$ and $\lambda\in (0,1)$. In particular, when $p\geqslant  1$ and $C:=\widehat C_v^+$, we have $\Gamma_\lambda(p)\not=\emptyset$ for every $\lambda\in (0,1)$ in view of Lemma \ref{lemma bdd above}.
\end{remark}

\begin{proof}
Let $p\geqslant  0$ be a fixed integer.
We argue by contradiction. Assume there exist sequences $\lambda_n\to 0^+$, $\gamma_n\in\Gamma_{\lambda_n}(p)$ and $t_n\to +\infty$ such that
\[
\lambda^{p+1}_n\alpha(\lambda_n,\gamma_n,t_n)\to +\infty\qquad\hbox{as $n\to +\infty$,}
\]
with
\[
\int_{-t_n}^0e^{-\lambda_n \int_s^0a(\gamma_n(\tau))d\tau}\Big(L\big(\gamma_n(s),\dot\gamma_n(s)\big)+c_0\Big)ds
\leqslant  \frac{C}{\lambda_n^p}+1.
\]
In particular, conditions \eqref{eq diverging alpha} and \eqref{leqCn} in Lemma \ref{lemma Cn} are satisfied with
\[
\theta_n:=\dfrac{C/\lambda_n^p+1}{\alpha(\lambda_n,\gamma_n,t_n)}.
\]
Let $\tilde{\mu}_n$ be the probability measure defined in (\ref{mun}). According to Lemma \ref{lemma Cn}, up to extraction of a subsequence (not relabeled), the measures $\tilde\mu_n$ weakly converge to a Mather
measure $\tilde \mu$.

Now we choose a subsolution $\varphi\leqslant  -1$ of (\ref{e0}). For every $\varepsilon>0$, there exists, in view of Theorem \ref{approx}, a function $\varphi_\varepsilon\in C^\infty(M)$ satisfying
\begin{equation}\label{vare}
  \|\varphi_\varepsilon-\varphi\|_\infty\leqslant  \varepsilon\qquad\hbox{and}\qquad
  H\big(x,D_x \varphi_\varepsilon\big)\leqslant  c_0+\varepsilon\quad\hbox{for all $x\in M$}.
\end{equation}
For $\varepsilon>0$ small, we have $\varphi_\varepsilon<0$. We get
\begin{flalign*}
\quad\  \theta_n
&
\geqslant  \int_{TM}\big(L(x,v)+c_0\big)\,d\tilde{\mu}_n(x,v)
&&
\\
&\geqslant  \int_{TM}\Big(  D_x\varphi_\varepsilon(v)-H\big(x,D_x \varphi_\varepsilon\big)+c_0\Big)\,d\tilde{\mu}_n(x,v)
\geqslant
\int_{TM}\big( D_x\varphi_\varepsilon(v)-\varepsilon\big) \,d\tilde{\mu}_n(x,v)
&&
\\
&
=\frac{\varphi_\varepsilon\big(\gamma_n(0)\big)-e^{-\lambda_n\int_{-t_n}^0a(\gamma_n(\tau))d\tau}\varphi_\varepsilon\big(\gamma_n(-t_n)\big)}{\alpha(\lambda_n,\gamma_n,t_n)}-\lambda_n\int_{TM}a(x)\varphi_\varepsilon(x) \,d\tilde{\mu}_n(x,v)-\varepsilon
&&
\\
&
> \frac{\varphi_\varepsilon\big(\gamma_n(0)\big)}{\alpha(\lambda_n,\gamma_n,t_n)}-\lambda_n\int_{TM}a(x)\varphi_\varepsilon(x)\,d\tilde{\mu}_n(x,v)-\varepsilon.
&&
\end{flalign*}
The equality appearing above is obtained via an integration by parts (cf. proof of Lemma \ref{lemma Cn} equation \eqref{IPP}, when we prove that $\tilde\mu$ is closed), while for the last inequality we have used that fact that $\varphi_\eps<0$.
Now we send $\varepsilon\to 0^+$ and divide the above inequality by $\lambda_n$.  We get
\[
\frac{C}{\lambda_n^{p+1}\alpha(\lambda_n,\gamma_n,t_n)} + \frac{1}{\lambda_n \alpha(\lambda_n,\gamma_n,t_n)}\geqslant  \frac{\varphi\big(\gamma_n(0)\big)}{\lambda_n\alpha(\lambda_n,\gamma_n,t_n)}-\int_{TM}a(x)\varphi(x)d\tilde{\mu}_n.
\]
Since $\lambda_n \alpha(\lambda_n,\gamma_n,t_n)  \geqslant  {\lambda}^{p+1}_n\alpha(\lambda_n,\gamma_n,t_n) \to +\infty$  as $n\to +\infty$, we get
\[
\int_{TM}a(x)\varphi(x)\,d\tilde\mu(x,v) \geqslant  0.
\]
Since for all  $m\geqslant  0$ the function $\varphi-m\leqslant -1$ is also a negative subsolution of (\ref{e0}), by replacing $\varphi$ with
$\varphi-m$ in the inequality above we obtain
\[
\|\varphi\|_\infty \|a\|_\infty
\geqslant
\int_{TM}a(x)\varphi(x)\,d\tilde{\mu}(x,v) \geqslant  m \int_{TM}a(x)\,d\tilde{\mu}(x,v),\qquad \hbox{for all $m\geqslant  0$,}
\]
which implies that $\int_{TM}a(x)d\tilde{\mu}\leqslant  0$. This contradicts assumption \eqref{V>0}.
\end{proof}

We now proceed to show that the value function $V_\lambda$ is bounded from below, for every fixed $\lambda\in \big(0,\lambda(1)\big)$, where $\lambda(1)>0$ is the value obtained according to Lemma \ref{lemma Cn} with $p=1$ and $C=\widehat C_v^+$, where $\widehat C_v^+$ is the constant provided by Lemma \ref{lemma bdd above}.

\begin{lemma}
There exists a constant $\widehat C^-_v>0$ independent of $\lambda$ such that
\[
V_\lambda(x)\geqslant  -\widehat C^-_v\lambda^{-2}\qquad\hbox{for all $x\in M$ and  $\lambda\in \big(0,\lambda(1)\big)$.}
\]
\end{lemma}
\begin{proof}
Let us fix $\gamma\in\Gamma_\lambda$. For every $t>0$ we have
\[\int_{-t}^0e^{-\lambda \int_s^0a(\gamma(\tau))d\tau}\Big(L\big(\gamma(s),\dot\gamma(s)\big)+c_0\Big)ds\geqslant  -C^-_L\int_{-t}^0e^{-\lambda \int_s^0a(\gamma(\tau))d\tau}ds,\]
where $C^-_L:=-\min\big\{\min_{TM}\big(L(x,v)+c_0\big),0\big\}\geqslant  0$.
We are now going to apply Lemma \ref{Ca} with $p=1$ and by choosing $C:=\widehat C_v^+$ in the definition of $\Gamma_\lambda(1)$, so that  $\Gamma_\lambda(1)=\Gamma_\lambda$: from the above inequality we infer
\[
\int_{-t}^0e^{-\lambda \int_s^0a(\gamma(\tau))d\tau}\Big(L\big(\gamma(s),\dot\gamma(s)\big)+c_0\Big)ds\geqslant  -\frac{C^-_L A(1)}{\lambda^2}
\qquad\hbox{for all $t>0$}.
\]
The assertion readily follows from the definition of $V_\lambda$.
\end{proof}

From the information gathered so far, we know that the value function $V_\lambda$ is finite-valued on $M$ for every $\lambda\in\big(0,\lambda(1)\big)$. We now proceed to get bounds for $V_\lambda$ from above and from below on $M$  independent of $\lambda$.

\begin{lemma}\label{lemma partial upper bound}
There is a constant $\overline{C}^+_v>0$ such that
\[
V_\lambda(x)\leqslant  \overline{C}^+_v\qquad\hbox{for all $x\in \A$ and  $\lambda\in \big(0,\lambda(1)\big)$.}
\]
\end{lemma}
\begin{proof}
According to \cite[Theorem 4.14 and Proposition 4.4]{gen}, see also \cite[Theorem 3.3]{DZ10}, for every $x\in\A$ there exists a curve
$\eta:\mathbb R\to \A$ with $\eta(0)=x$ such that, for every subsolution $\varphi$ of \eqref{e0},
\begin{equation}\label{eq static curve}
L\big(\eta(s),\dot\eta(s)\big)+c_0=\frac{d}{ds}(\varphi\comp\eta)(s)\quad \hbox{for a.e.\ $s\in\mathbb R$.}
\end{equation}
Let us denote by $\K$ the family of curves $\eta:\mathbb R\to \A$ satisfying \eqref{eq static curve}. We remark for further use that these curves are equi-Lipschitz, see for instance \cite[Lemma 4.9]{gen}.

Pick a subsolution $\varphi\leqslant   0$ of \eqref{e0} and fix $x\in\A$. From the definition of $V_\lambda$ we get
\begin{eqnarray}
V_\lambda(x)
&\leqslant &
\limsup_{t\to+\infty}\int_{-t}^0 e^{-\lambda \int_s^0a(\eta(\tau))d\tau}\Big(L\big(\eta(s),\dot\eta(s)\big)+c_0\Big)ds
\nonumber\\
&=& \limsup_{t\to+\infty} \int_{-t}^0 e^{-\lambda \int_s^0 a(\eta(\tau))\, d\tau}\,\frac{d}{ds}(\varphi\comp\eta)(s)\, ds \label{VA}\\
&=&\limsup_{t\to+\infty}\bigg\{\varphi\big(\eta(0)\big)-e^{-\lambda\int_{-t}^0a(\eta(\tau))d\tau}\varphi\big(\eta(-t)\big)-\lambda\alpha(\lambda,\eta,t)\int_{TM}a(x)\varphi(x)d\tilde{\mu}^\eta_t\bigg\},\nonumber
\end{eqnarray}
where $\tilde{\mu}^\eta_t\in\parts (TM)$ is defined by
\[
\int_{TM}f(x,v) \,d \tilde{\mu}^\eta_t(x,v):=\frac{\int_{-t}^0e^{-\lambda\int_s^0a(\eta(\tau))d\tau}f\big(\eta(s),\dot{\eta}(s)\big)ds}{\int_{-t}^0e^{-\lambda\int_s^0a(\eta(\tau))d\tau}ds},\qquad \forall f\in C_c(TM).
\]
The second equality in \eqref{VA} is derived via an integration by parts as for \eqref{IPP}.
By assumption \eqref{V>0} and compactness of the family of Mather measures $\Mis$, there exists
$\varepsilon>0$ such that
\begin{equation}\label{equivalent V>0}
\int_{TM} a(x)\,d\tilde\mu(x,v)>\varepsilon
\qquad
\hbox{for all $\tilde\mu\in\Mis$}.
\end{equation}
We show that there is $T_0>0$ such that, for all  curves $\eta\in\K$, we have
\[\int_{-t}^0a\big(\eta(s)\big)\,ds>\varepsilon t\qquad \hbox{for all $t\geqslant  T_0$.}
\]
We argue by contradiction. Assume there exist sequences $\eta_n\in\K$ and $t_n\to 0$ such that
\begin{equation}\label{tn<e}
  \int_{-t_n}^0a\big(\eta_n(s)\big) \,ds\leqslant  \varepsilon t_n.
\end{equation}
Define $\tilde \mu_n \in \parts(TM)$ by
\[
\int_{TM}f(x,v)d \tilde{\mu}_n :=\frac{1}{t_n}\int_{-t_n}^0f\big(\eta_n(s),\dot\eta_n(s)\big)\,ds,\quad \forall f\in C_c(TM).
\]
Due to the fact that the curves $(\eta_n)_n$ are equi-Lipschitz, the measures $(\tilde\mu_n)_n$ have equi-compact support. In particular, up to extracting a subsequence (not relabeled), they weakly converge to a probability measure $\tilde\mu\in\parts(TM)$. Furthermore
\[
\lim_n \int_{TM} f(x,v)\,d\tilde\mu_n(x,v)=\int_{TM} f(x,v)\,d\tilde\mu(x,v)
\qquad
\forall f\in C(TM).
\]
We claim that $\tilde\mu$ is closed.
Indeed, for every  $\phi\in C^1(M)$ we have
\begin{eqnarray*}
\int_{TM }   D_x\phi(v)\, d\tilde{\mu}(x,v)
&=&
\lim_{n} \int_{TM}  D_x\phi(v)\, d\tilde{\mu}_n(x,v)\\
&=&
\lim_{n} \frac{1}{t_n}\int_{-t_n}^0 \frac{d}{ds}(\phi\comp\eta_n)(s)\,ds
\leqslant
\lim_{n} \frac{2\|\phi\|_\infty}{t_n}= 0.
\end{eqnarray*}
We proceed to show that $\tilde\mu$ is minimizing, namely, a Mather measure. Pick a subsolution $\varphi$ to \eqref{e0}. By exploiting \eqref{eq static curve}, we get
\begin{flalign*}
\int_{TM}\big(L(x,v)+c_0\big) d\tilde{\mu}
&
=
\lim_n  \int_{TM}\big(L(x,v)+c_0\big) d\tilde{\mu}_n=\lim_n \frac{1}{t_n}\int_{-t_n}^0\!\!\Big(L\big(\eta_n(s),\dot\eta_n(s)\big)+c_0\Big)ds
&&
\\
&
=
\lim_{n} \frac{1}{t_n}\int_{-t_n}^0 \frac{d}{ds}(\varphi\comp\eta_n)(s)\,ds
=
\lim_n
\frac{\varphi\big(\eta_n(0)\big)-\varphi\big(\eta_n(-t_n)\big)}{t_n}
=
0.
&&
\end{flalign*}
By (\ref{tn<e}), we also have
$\int_{TM}a(x)\,d\tilde\mu(x,v)\leqslant  \varepsilon,$
which leads to a contradiction with \eqref{equivalent V>0}.
Then for $t>T_0$ we have
\[
e^{-\lambda\int_{-t}^0a(\eta(\tau))d\tau}\leqslant  e^{-\lambda\varepsilon t}
\]
and
\begin{eqnarray*}
0<\lambda\alpha(\lambda,\eta,t)
&=&
\lambda \int_{-t}^0e^{-\lambda \int_s^0a(\eta(\tau))d\tau}ds\\
&=&\lambda \int_{-T_0}^0e^{-\lambda \int_s^0a(\eta(\tau))d\tau}ds+\lambda\int_{-t}^{-T_0}e^{-\lambda \int_s^0a(\eta(\tau))d\tau}ds\\
&\leqslant & \lambda \int_{-T_0}^0 e^{-\lambda \|a\|_\infty s}\, ds+\lambda\int_{-t}^{-T_0}e^{\lambda \varepsilon s}ds\\
&\leqslant
&\frac{ e^{\lambda {\|a\|_\infty}T_0}-1}{\|a\|_\infty}+\frac{e^{-\lambda\varepsilon T_0}}{\varepsilon}
\leqslant
\frac{e^{{\|a\|_\infty} T_0}}{\|a\|_\infty}+\frac{1}{\varepsilon}=:C_\eta.
\end{eqnarray*}
By using these inequalities in \eqref{VA} and recalling that $\varphi\leqslant  0$, we finally get
\[V_\lambda(x)\leqslant  C_\eta\|a\|_\infty \|\varphi\|_\infty,\]
which gives the upper bound of $V_\lambda$ on $\mathcal A$ independent of $\lambda$.
\end{proof}

By exploiting the fact that $V_\lambda$ satisfies the Dynamic Programming Principle, we show that the partial upper bound obtained in Lemma \ref{lemma partial upper bound} actually entails a uniform upper bound on the whole $M$.

\begin{proposition}\label{prop tildeCv}
There is $C_v^+>0$ independent of $\lambda$ such that
\[
V_\lambda(x)\leqslant  C_v^+\qquad\hbox{for all $x\in M$ and $\lambda\in \big(0,\lambda(1)\big)$.}
\]
\end{proposition}
\begin{proof}
Fix $x\in M$ and pick a point $y\in\mathcal A$. Set $ d :=d(x,y)$,  $\overline{D}:=$diam$(M)$ and $C_L:=\max\limits_{x\in M,\|v\|_x\leqslant  1}|L(x,v)+c_0|$ . Take a geodesic $\zeta:[- d ,0]\to M$ with $\zeta(- d )=y$, $\zeta(0)=x$ and $\|\dot{\zeta}\|_\zeta=1$.  By
Proposition \ref{prop DPP}  we have
\begin{align*}
V_\lambda(x)&\leqslant  e^{-\lambda \int_{- d }^0a(\zeta(\tau))d\tau}V_\lambda(y)+\int_{- d }^0 e^{-\lambda \int_s^0a(\zeta(\tau))d\tau}\Big(L\big(\zeta(s),\dot\zeta(s)\big)+c_0\Big)ds
\\ &\leqslant  e^{\lambda \|a\|_\infty  d }\overline{C}_v^+ +C_L de^{\lambda \|a\|_\infty  d }
\leqslant
e^{\|a\|_\infty \overline{D}}(\overline{C}_v^+ +C_L \overline{D}),
\end{align*}
which gives the sought uniform upper bound of $V_\lambda$.
\end{proof}

Now that we know that $V_\lambda$ is uniformly bounded from above, we can prove that
$V_\lambda$ is also uniformly bounded from below.

\begin{proposition}\label{prop lower bound}
There exist $\overline \lambda\in \big(0,\lambda(1)\big)$ and a constant $C^-_v>0$ independent of $\lambda$ such that
\[
V_\lambda(x)\geqslant  -C_v^-\qquad\hbox{for all $x\in M$ and $\lambda\in (0,\overline\lambda)$.}
\]
In particular, $|V_\lambda(x)|\leqslant  C_v:=\max\{C_v^-,C_v^+\}$\ for all $x\in M$ and $\lambda\in (0,\overline\lambda)$.
\end{proposition}

\begin{proof}
By Proposition \ref{prop tildeCv}, we know that, for $\lambda<\lambda(1)$, the infimum in $V_\lambda$ is taken among the set
\begin{eqnarray*}
\Gamma:=\bigg\{\gamma:(-\infty,0]\to M:\ \limsup_{t\to+\infty}\int_{-t}^0e^{-\lambda \int_s^0a(\gamma(\tau))d\tau}\Big(L\big(\gamma(s),\dot\gamma(s)\big)+c_0\Big)ds\leqslant  {C}_v^+ \bigg\}.
\end{eqnarray*}
By Lemma \ref{Ca}, there exists $\lambda(0) \in (0,1)$ (depending on $C^+_v$) such that
\begin{equation*} 
 A(0) :=\sup\bigg\{\lambda\alpha(\lambda,\gamma,t):\ \lambda\in\big(0,\lambda(0)\big),\ \gamma\in \Gamma, \ t>0\,\bigg\}<+\infty.
\end{equation*}
Let us set $\overline \lambda:=\min\{\lambda(0),\lambda(1)\}$.\footnote{We recall that $\lambda(1)\in (0,1)$ is the value obtained according to Lemma \ref{lemma Cn} with $p=1$ and $C=\widehat C_v^+$, where $\widehat C_v^+$ is the constant provided by Lemma \ref{lemma bdd above}.}
Let $\varphi\leqslant  -1$ be a subsolution of (\ref{e0}). For each $\varepsilon>0$, we can take $\varphi_\varepsilon\in C^\infty(M)$ satisfying $\|\varphi_\varepsilon-\varphi\|_\infty\leqslant  \varepsilon$ and $H\big(x,D_x \varphi_\varepsilon\big)\leqslant  c_0+\varepsilon$\ for all $x\in M$, given by Theorem \ref{approx}.
For $\varepsilon>0$ small, we have $\varphi_\varepsilon<0$. For each $\gamma\in \Gamma$, we have
\begin{align*}
&\int_{-t}^0e^{-\lambda \int_s^0a(\gamma(\tau))d\tau}\Big(L\big(\gamma(s),\dot\gamma(s)\big)+c_0\Big)ds
\\ &\geqslant  \int_{-t}^0e^{-\lambda \int_s^0a(\gamma(\tau))d\tau}\Big( D_{\gamma(s)}\varphi_\varepsilon\big(\dot{\gamma}(s)\big)-H\big(\gamma(s),D_{\gamma(s)} \varphi_\varepsilon\big)+c_0\Big)ds
\\ &\geqslant  \int_{-t}^0e^{-\lambda \int_s^0a(\gamma(\tau))d\tau}\bigg(\frac{d\varphi_\varepsilon}{dt}\big(\gamma(s)\big)-\varepsilon\bigg)ds
\\ &=\varphi_\varepsilon\big(\gamma(0)\big)-e^{-\lambda\int_{-t}^0a(\gamma(\tau))d\tau}\varphi_\varepsilon\big(\gamma(-t)\big)-\alpha(\lambda,\gamma,t)\int_{TM}\big(\lambda a(x)\varphi_\varepsilon(x)+\varepsilon\big)d\tilde{\mu}^\gamma_t
\\ &\geqslant  \varphi_\varepsilon\big(\gamma(0)\big)-\alpha(\lambda,\gamma,t)\int_{TM}\big(\lambda a(x)\varphi_\varepsilon(x)+\varepsilon\big)d\tilde{\mu}^\gamma_t,
\end{align*}
where
\[
\int_{TM} f(x,v) d\tilde{\mu}^\gamma_t(x,v):=\frac{\int_{-t}^0e^{-\lambda\int_s^0a(\gamma(\tau))d\tau}f\big(\gamma(s),\dot{\gamma}(s)\big)ds}{\alpha(\lambda,\gamma,t)},
\qquad \forall f\in C_c(TM).
\]
We recall that $\alpha(\lambda,\gamma,t):=\int_{-t}^0e^{-\lambda\int_s^0a(\gamma(\tau))d\tau}ds$.
Now let $\varepsilon\to 0^+$ to get
\[
\int_{-t}^0e^{-\lambda \int_s^0a(\gamma(\tau))d\tau}\Big(L\big(\gamma(s),\dot\gamma(s)\big)+c_0\Big)ds\geqslant  \varphi\big(\gamma(0)\big)-\lambda\alpha(\lambda,\gamma,t)\int_{TM}a(x)\varphi(x)d\tilde{\mu}^\gamma_t.
\]
Sending $t\to+\infty$ we get
\[\limsup_{t\to+\infty}\int_{-t}^0e^{-\lambda \int_s^0a(\gamma(\tau))d\tau}\Big(L\big(\gamma(s),\dot\gamma(s)\big)+c_0\Big)ds
\geqslant  -\big(1+A(0)\|a\|_\infty\big)\|\varphi\|_\infty.\]
The bound from below readily follows from this by definition of $V_\lambda$.
The last assertion is a consequence of Proposition \ref{prop tildeCv}  and of what we have just shown above.
\end{proof}

We now proceed to show that the value function is Lipschitz continuous. This is indeed a consequence of this more general result.

\begin{proposition}\label{prop vlip}
Let $w:M\to\mathbb R$ be a bounded function. Assume that $w$ satisfies the Dynamic Programming Principle, i.e., for each absolutely continuous curve $\gamma:[-t,0]\to M$, we have
\begin{flalign}\label{w-w}\tag{DPP}
 & \ w\big(\gamma(0)\big)-e^{-\lambda \int_{-t}^0 a(\gamma(\tau))d\tau}w\big(\gamma(-t)\big)\leqslant  \int_{-t}^0 e^{-\lambda \int_s^0a(\gamma(\tau))d\tau}\Big(L\big(\gamma(s),\dot\gamma(s)\big)+c_0\Big)ds.&&
\end{flalign}
Then $w$ is Lipschitz continuous, with a Lipschitz constant that only depends on $L,\,\D{diam}(M)$ and $\|w\|_\infty$.
\end{proposition}

\begin{proof}
Pick $x,y\in M$. Let $\zeta :[- d ,0]\to M$ be a geodesic with $\zeta (- d )=y$ and $\zeta (0)=x$, where $ d :=d(x,y)$. By (\ref{w-w}) we have
\begin{eqnarray*}
w(x)-w(y)&\leqslant  -w(y)\bigg(1-e^{-\lambda\int_{- d }^0a(\zeta (\tau))d\tau}\bigg)+\int_{- d }^0e^{-\lambda\int_{s}^0a(\zeta (\tau))d\tau}\big(L\big(\zeta (s),\dot{\zeta }(s)\big)+c_0\Big)ds.
\end{eqnarray*}
There is $\tau_0 \in(- d ,0)$ such that $\int_{- d }^0a\big(\zeta (\tau)\big)d\tau=a\big(\zeta (\tau_0 )\big) d $. If $a\big(\zeta (\tau_0 )\big)\geqslant  0$, we have
\[0\leqslant  1-e^{-\lambda a(\zeta (\tau_0 )) d }\leqslant  e^{\lambda a(\zeta (\tau_0 )) d }-1\leqslant  e^{\lambda\|a\|_\infty  d }-1.\]
If $a\big(\zeta (\tau_0 )\big)< 0$, we have
\[0\geqslant  1-e^{-\lambda a(\zeta (\tau_0 )) d }\geqslant  1-e^{\lambda \|a\|_\infty  d }.\]
Then
\begin{eqnarray*}
w(x)-w(y)\leqslant  \lambda \|w\|_\infty\frac{e^{\lambda\|a\|_\infty  d }-1}{\lambda}+C_L\int_{- d }^0e^{-\lambda \|a\|_\infty s}ds\leqslant  (\|a\|_\infty\|w\|_\infty+C_L)\frac{e^{\lambda\|a\|_\infty d}-1}{\lambda \|a\|_\infty}.
\end{eqnarray*}
Since $\lambda\in(0,1)$ and $ d \leqslant  \overline{D}$, there is $C_{\overline{D}}>0$ such that $e^{\lambda \|a\|_\infty d }-1\leqslant  C_{\overline{D}}\lambda\tilde  d$. Exchanging the role of $x$ and $y$, we get the conclusion.
\end{proof}

As a consequence of the previous proposition and Proposition \ref{prop lower bound}, we derive the following information.

\begin{corollary}\label{cor equi-Lipschitz}
There is $\kappa>0$ independent of $\lambda$ such that the functions $\{V_\lambda\,:\,\lambda\in(0,\bar \lambda)\}$ are
$\kappa$-Lipschitz continuous.
\end{corollary}

Next, we show that the value function is a viscosity subsolution of the equation \eqref{VH}. Indeed, the following result holds.

\begin{proposition}\label{prop wiff}
Let $w\in C(M)$. Then $w$ is a subsolution of (\ref{VH}) if and only if (\ref{w-w}) holds for every absolutely curve $\gamma:[-t,0]\to M$ and every $t>0$.
\end{proposition}

\begin{proof}
Let us first assume that $w$ is a viscosity subsolution of \eqref{e0}.
By Proposition \ref{prop equivalence}, we derive that $w$ is Lipschitz continuous.
Using Theorem \ref{approx}, we take a sequence $w_n\in C^1(M)$ such that $\|w_n-w\|_\infty\leqslant  1/n$ and

\[\lambda a(x)w_n(x)+H\big(x,D_x w_n\big)\leqslant  c_0+\frac{1}{n}\qquad\hbox{for all $x\in M$}.
\]
For $\gamma:[-t,0]\to M$, we have, performing the now usual integration by parts and the Fenchel inequality \eqref{fenchel},
\begin{flalign*}
\qquad&
w_n\big(\gamma(0)\big)-e^{-\lambda \int_{-t}^0a(\gamma(\tau))d\tau}w_n\big(\gamma(-t)\big)
=\int_{-t}^0\frac{d}{ds}\bigg(e^{-\lambda \int_{s}^0a(\gamma(\tau))d\tau}w_n\big(\gamma(s)\big)\bigg)ds
&&
\\
&
=\int_{-t}^0e^{-\lambda \int_{s}^0a(\gamma(\tau))d\tau}\bigg(\lambda a\big(\gamma(s)\big)w_n\big(\gamma(s)\big)+ D_{\gamma(s)}w_n\big(\dot\gamma(s)\big)\bigg)ds
&&
\\
&
\leqslant  \int_{-t}^0e^{-\lambda \int_{s}^0a(\gamma(\tau))d\tau}\bigg(\lambda a\big(\gamma(s)\big)w_n\big(\gamma(s)\big)+H\big(\gamma(s),
D_{\gamma(s)}w_n\big)+L\big(\gamma(s),\dot\gamma(s)\big)\bigg)ds
&&
\\
&
\leqslant  \int_{-t}^0e^{-\lambda \int_{s}^0a(\gamma(\tau))d\tau}\bigg(c_0+\frac{1}{n}+L\big(\gamma(s),\dot\gamma(s)\big)\bigg)ds
&&
\\
&
\leqslant  \frac{1}{n}\int_{-t}^0e^{-\lambda \|a\|_\infty s}ds+\int_{-t}^0e^{-\lambda \int_{s}^0a(\gamma(\tau))d\tau}\Big(L\big(\gamma(s),\dot\gamma(s)\big)+c_0\Big)ds
&&
\\
&=\frac{e^{\lambda\|a\|_\infty t}-1}{n\|a\|_\infty \lambda}+\int_{-t}^0 e^{-\lambda \int_s^0a(\gamma(\tau))d\tau}\Big(L\big(\gamma(s),\dot\gamma(s)\big)+c_0\Big)ds.
&&
\end{flalign*}
The assertion follows by sending $n\to+\infty$.

Conversely, let us assume that $w$ satisfies (\ref{w-w}). According to Proposition \ref{prop vlip}, $w$ is Lipschitz continuous. We only need to check if $w$ is a subsolution where $w$ is differentiable, thanks to Proposition \ref{prop equivalence}.
Let $w$ be differentiable at $x$. We take a $C^1$ curve $\gamma:[-t,0]\to M$ with $\gamma(0)=x$ and $\dot\gamma(0)=v$. Then
\[\frac{w\big(\gamma(0)\big)-e^{-\lambda \int_{-t}^0a(\gamma(\tau))d\tau}w\big(\gamma(-t)\big)}{t}\leqslant  \frac{1}{t}\int_{-t}^0 e^{-\lambda \int_s^0a(\gamma(\tau))d\tau}\Big(L\big(\gamma(s),\dot\gamma(s)\big)+c_0\Big)ds.\]
Letting $t\to 0^+$, we get
\[\lambda a(x)w(x)+ D_x w(v) -L(x,v)\leqslant  c_0.\]
Taking the supremum with respect to $v$, we get, thanks to \eqref{convconj},
\[\lambda a(x)w(x)+H(x,D_x w)\leqslant  c_0,\]
which implies that $w$ is a subsolution.
\end{proof}

We proceed to show that the value function is the maximal viscosity subsolution of \eqref{VH}, and hence a solution by maximality. We need an auxiliary lemma first.

\begin{lemma}\label{leto0}
Let $\lambda>0$  and let $\gamma:(-\infty,0]\to M$ be an absolutely continuous curve. Let us assume that
\begin{equation}\label{al<inf}
  \sup_{t>0}\alpha(\lambda,\gamma,t)<+\infty.
\end{equation}
Then
$e^{-\lambda \int_{-t}^0 a(\gamma(\tau))d\tau}\to 0$\  as $t\to+\infty$.
\end{lemma}
\begin{proof}
We argue by contradiction. Assume there exist an increasing sequence $t_n\to+\infty$ and a $\delta\in (0,1)$ small enough such that
\[
e^{-\lambda \int_{-t_n}^0 a(\gamma(\tau))d\tau}\geqslant  \delta
\qquad
\qquad\hbox{for all $n\in\N$}.
\]
Then, for all $t\in\big(t_n,t_n+\frac{\ln(\delta^{-1})}{\lambda\|a\|_\infty}\big)$, we have
\[
-\lambda\int_{-t}^{-t_n}a\big(\gamma(\tau)\big)d\tau\geqslant  -\lambda\|a\|_\infty|t-t_n|\geqslant  \ln\delta,
\]
hence
\[
e^{-\lambda \int_{-t}^0 a(\gamma(\tau))d\tau}=e^{-\lambda \int_{-t}^{-t_n} a(\gamma(\tau))d\tau}e^{-\lambda \int_{-t_n}^0 a(\gamma(\tau))d\tau}\geqslant  \delta^2 , \qquad \hbox{for all $n\in\mathbb N$.}
\]
Let us pick $r>0$ with $r\leqslant \frac{\ln(\delta^{-1})}{\lambda\|a\|_\infty}$. Up to extracting a subsequence, we can assume that $r\leqslant  |t_{n+1}-t_n|$ for all $n\in\mathbb N$. Let us set $t_0:=0$.
We have
\[\int_{-t_n}^0e^{-\lambda \int_{s}^0 a(\gamma(\tau))d\tau}ds\geqslant  \sum_{i=0}^{n-1}\int_{-t_i-r}^{-t_i}e^{-\lambda \int_{s}^0 a(\gamma(\tau))d\tau}ds\geqslant  r\delta^2n\to +\infty\]
as $n\to+\infty$, which contradicts (\ref{al<inf}).
\end{proof}

Let us prove the result announced above. We recall that the $\bar \lambda$ appearing in the next two statements is the real number in $(0,1)$ provided by Proposition \ref{prop lower bound}.

\begin{proposition}\label{prop maximal subsolution}
For every fixed $\lambda\in (0,\bar \lambda)$, the value function $V_\lambda$ is the maximal subsolution of (\ref{VH}). In particular, it is a viscosity solution of (\ref{VH}).
\end{proposition}
\begin{proof}
Let $w$ be a subsolution of (\ref{VH}). Let us fix $x\in M$.
By definition, for all $\varepsilon>0$, there is $\gamma\in\Gamma$ such that
\begin{align*}
V_\lambda(x)+\varepsilon&>\limsup_{T\to+\infty}\int_{-T}^0e^{-\lambda \int_s^0a(\gamma(\tau))d\tau}\Big(L\big(\gamma(s),\dot\gamma(s)\big)+c_0\Big)ds
\\ &=\int_{-t}^0 e^{-\lambda \int_s^0a(\gamma(\tau))d\tau}\Big(L\big(\gamma(s),\dot\gamma(s)\big)+c_0\Big)ds
\\ &\quad +e^{-\lambda \int_{-t}^0 a(\gamma(\tau))d\tau}\limsup_{T\to+\infty}\int_{-T}^0e^{-\lambda \int_s^0a(\xi (\tau))d\tau}\Big(L\big(\xi(s),\dot{\xi}(s)\big)+c_0\Big)ds,
\end{align*}
where $\xi(s):=\gamma(s-t)$ for $s\leqslant  0$. By Proposition \ref{prop wiff}, we have
\[\int_{-t}^0 e^{-\lambda \int_s^0a(\gamma(\tau))d\tau}\Big(L\big(\gamma(s),\dot\gamma(s)\big)+c_0\Big)ds\geqslant  w(x)-e^{-\lambda \int_{-t}^0 a(\gamma(\tau))d\tau}w\big(\gamma(-t)\big).\]
By the definition of $V_\lambda$, we have
\begin{eqnarray*}
e^{-\lambda \int_{-t}^0 a(\gamma(\tau))d\tau}\limsup_{T\to+\infty}\int_{-T}^0e^{-\lambda \int_s^0a(\xi(\tau))d\tau}\Big(L\big(\xi(s),\dot{\xi}(s)\big)+c_0\Big)ds\geqslant  e^{-\lambda \int_{-t}^0 a(\gamma(\tau))d\tau}V_\lambda\big(\gamma(-t)\big).
\end{eqnarray*}
Combining all the above inequalities, we conclude
\[
V_\lambda(x)+\varepsilon>w(x)+e^{-\lambda \int_{-t}^0 a(\gamma(\tau))d\tau}\Big(V_\lambda\big(\gamma(-t)\big)-w\big(\gamma(-t)\big)\Big).
\]
We know that $\sup\limits_{t>0} \alpha(\lambda,\gamma,t)<+\infty$, cf. proof of Proposition \ref{prop lower bound}.
By Lemma \ref{leto0}, we derive that
$e^{-\lambda \int_{-t}^0 a(\gamma(\tau))d\tau}\to 0$
as $t\to+\infty$. By finally sending $\varepsilon\to 0^+$ we conclude that $V_\lambda(x)\geqslant  w(x)$. This, together with Propositions \ref{prop DPP} and \ref{prop wiff}, proves the asserted maximality of $V_\lambda$.

Now we show that $V_\lambda$ is a viscosity solution of (\ref{VH}). {Since $V_\lambda$ is the maximal subsolution, we only need to show that it is a supersolution. The argument is standard and depends on the bump construction. We give it below for the reader's convenience.} Assume, by contradiction, that the supersolution test fails at some point $z\in M$. This means that there is a strict subtangent $\phi\in C^1(M)$ of $V_\lambda$ at $z$\,\footnote{Meaning that $V_\lambda-\phi$ has a strict local minimum at $z$.} such that
\[\lambda a(z)V_\lambda(z)+H\big(z,D_z\phi \big)<c_0.
\]
Up to adding a constant to $\phi$, we can assume that $V_\lambda(z)=\phi(z)$.
Let $B_r (z)$ be the open ball centered at $z$ with the radius $r $. Let us choose $r >0$ and  $\eps>0$ small enough so that
\begin{eqnarray}\label{eq bump}
\phi+\eps<V_\lambda\ \hbox{on $\partial B_r (z)$}\quad\hbox{and}\quad \lambda a(x)(\phi(x)+\varepsilon)+H\big(x,D_x \phi\big)<c_0 \ \  \forall x\in B_r (z).
\end{eqnarray}
Set
\begin{equation*}
  \widetilde V_\lambda(x):=
\begin{cases}
   \max\{V_\lambda(x),\phi(x)+\varepsilon\} & \hbox{if $x\in B_r (z)$.}\\
   V_\lambda(x) & \hbox{if  $x\in M\setminus B_r (z)$}
\end{cases}
\end{equation*}
Due to \eqref{eq bump}, the function $\widetilde V_\lambda$ is a subsolution of \eqref{VH} in $B_r (z)$ as the maximum of two subsolutions in that open ball, and it agrees with  $V_\lambda$ in an open neighborhood of $M\setminus B_r (z)$. This readily implies that $\widetilde V_\lambda$ is a subsolution of \eqref{VH} on the whole $M$. Yet, we have
$\widetilde V_\lambda(z)=\varphi(z)+\varepsilon>V_\lambda(z)$, contradicting the fact that $V_\lambda$ is the maximal subsolution of (\ref{VH}). This shows that $V_\lambda$ is indeed a solution to \eqref{VH} in $M$.
\end{proof}

From now on, we denote by $u_\lambda(x)$ the value function $V_\lambda(x)$, since it is the maximal solution.

\begin{proposition}
There exists $\lambda_0\in (0,\bar\lambda)$ such that, for every fixed $x\in M$ and $\lambda\in (0,\lambda_0)$, we can find a curve $\gamma^x_\lambda:(-\infty,0]\to M$ with $\gamma^x_\lambda(0)=x$ such that
\[u_\lambda(x)=\int_{-\infty}^0e^{-\lambda \int_s^0a(\gamma^x_\lambda(\tau))d\tau}\Big(L\big(\gamma^x_\lambda(s),\dot\gamma^x_\lambda(s)\big)+c_0\Big)ds.\]
Furthermore, the curve $\gamma^x_\lambda$ is $\hat \kappa$-Lipschitz continuous, for some constant $\hat \kappa>0$ independent of
$\lambda\in (0,\lambda_0)$ and $x\in M$.
\end{proposition}
\begin{proof}
Let us fix $\lambda\in (0,\bar\lambda)$. According to Proposition \ref{prop maximal subsolution},
$u_\lambda$ is a viscosity, hence a weak KAM, solution of
\[\lambda a(x)u_\lambda(x)+H(x,D_x u)=c_0\qquad\hbox{in $M$}.
\]
By Lemma \ref{gameq}, we know that, for every fixed $x\in M$,  there is a Lipschitz curve $\gamma^x_\lambda:(-\infty,0]\to M$ with $\gamma^x_\lambda(0)=x$ such that
\begin{eqnarray}\label{eq calibrating}
u_\lambda\big(\gamma^x_\lambda(b_2)\big)-u_\lambda\big(\gamma^x_\lambda(b_1)\big)=\int_{b_1}^{b_2}\bigg[L\big(\gamma^x_\lambda(s),\dot{\gamma}^x_\lambda(s)\big)+c_0-\lambda a\big(\gamma^x_\lambda(s)\big)u_\lambda\big(\gamma^x_\lambda(s)\big)\bigg]ds
\end{eqnarray}
for all $b_1<b_2\leqslant  0$.
The above equality with $b_1=-t$ and $b_2=0$ can be restated as
\begin{equation}\label{ueq}
  u(t)=u(0)+\int_0^t h(s)u(s)ds+\int_0^t\ell(s)ds
  \qquad
  \hbox{for all $t>0$},
\end{equation}
where
\[
u(s):=-u_\lambda\big(\gamma^x_\lambda(-s)\big),\quad h(s):=\lambda a\big(\gamma^x_\lambda(-s)\big),\quad \ell(s):=L\big(\gamma^x_\lambda(-s),\dot{\gamma}^x_\lambda(-s)\big)+c_0.
\]
Note that the function $u$ and $h$ are in $L^\infty\big([0,t]\big)$. Since $\ell$ is uniformly bounded from below, this implies, in view of \eqref{ueq}, that $\ell$ is in $L^1\big([0,t]\big)$.
The operator $\mathcal T$ defined by
\[\mathcal Tf(t)=u(0)+\int_0^t h(s)f(s)ds+\int_0^t\ell(s)ds\]
is a contraction when $\lambda \|a\|_\infty t<1$, in particular there is a unique fixed point of $\mathcal T$. This gives the existence of the local unique solution of (\ref{ueq}), which is defined, for example, on $[0,t_0]$ with $t_0=\frac{1}{2\lambda \|a\|_\infty}$. Since the time of local existence is independent of the initial data, the maximal solution is defined for all $t>0$, and is given by
\[u(t)=u(0)e^{\int_0^th(s)ds}+\int_0^t\ell(s)e^{\int^t_{s}h(\tau)d\tau}ds,\]
as an explicit computation shows.
This gives directly
\begin{flalign}\label{ulam}
&  u_\lambda(x)=e^{-\lambda \int_{-t}^0a(\gamma^x_\lambda(\tau))d\tau}u_\lambda\big(\gamma^x_\lambda(-T)\big)+\!\!\int_{-t}^0 e^{-\lambda \int_s^0a(\gamma^x_\lambda(\tau))d\tau}\Big(L\big(\gamma^x_\lambda(s),\dot\gamma^x_\lambda(s)\big)+c_0\Big)ds&&
\end{flalign}
for all $t>0$. Due to the fact that the functions $(u_\lambda)_{\lambda\in (0,\overline\lambda)}$ are equi-bounded and equi-Lipschitz, this
implies that the curve $\gamma_\lambda^x$ is $\hat \kappa$-Lipschitz, with a Lipschitz constant
$\hat\kappa$ that is independent of $\lambda\in (0,\overline\lambda)$ and $x\in M$, see Lemma \ref{gameq}.

We want to show that there exists $\lambda_0\in (0,\bar\lambda]$ such that
$e^{-\lambda \int_{-t}^0a(\gamma^x_\lambda(\tau))d\tau}\to 0$ as $t\to+\infty$ whenever $\lambda\in (0,\lambda_0)$. According to Lemma \ref{leto0}, it suffices to show that there exists a $\lambda_0\in (0,\bar\lambda]$ such that
\begin{equation}\label{eq L^1 estimate}
\sup\bigg\{\lambda\alpha(\lambda,\gamma,t):\  \lambda\in (0,\lambda_0), \gamma\in\mathscr C_\lambda^x,\, x\in M,\,\ t>0\bigg\}<+\infty,
\end{equation}
where we have denoted by $\mathscr C^x_\lambda$ the family of absolutely continuous curves $\gamma:(-\infty,0]\to M$ with $\gamma(0)=x$ that satisfy \eqref{eq calibrating}. Notice in fact that \eqref{eq L^1 estimate} implies in particular that, for every fixed
$\lambda\in (0,\lambda_0)$, condition \eqref{al<inf} in Lemma \ref{leto0} is met.
We argue by contradiction. Let us assume the claim false. Then there exist sequences $\lambda_n\to 0^+$, $t_n\in (0,+\infty)$, $x_n\in M$  and
$\gamma_n\in \mathscr C_{\lambda_n}^{x_n}$ such that $\lambda_n \alpha(\lambda_n,\gamma_n,t_n)\to +\infty$ as $n\to +\infty$. Notice that the latter implies that $t_n\to +\infty$ and $\alpha_n:=\alpha(\lambda_n,\gamma_n,t_n)\to +\infty$ as $n\to +\infty$.
Let $\tilde{\mu}_n$ be the probability measure defined in (\ref{mun}). Then (\ref{ulam}) can be restated as
\begin{eqnarray}\label{e/e}
  \int_{TM}\big(L(x,v)+c_0\big)\,d\tilde{\mu}_n(x,v)=\frac{1}{\alpha_n}\bigg(u_{\lambda_n}\big(\gamma_n(0)\big)-e^{-\lambda_n \int_{-t_n}^0a(\gamma_n(\tau))d\tau}u_{\lambda_n}\big(\gamma_n(-t_n)\big)\bigg)
\end{eqnarray}
for all $n\in\N$. With the aid of Lemma \ref{a/b}, it is easy to check that the right-hand side of \eqref{e/e} goes to $0$ as $n\to +\infty$.
We can therefore apply Lemma \ref{lemma Cn} to infer that, up to extracting a subsequence, $\tilde{\mu}_n$ weakly converges to a Mather measure $\tilde{\mu}$.

Now we choose a subsolution $\varphi\leqslant  -1$ of (\ref{e0}). For $\varepsilon>0$, we take $\varphi_\varepsilon\in C^\infty(M)$ satisfying $\|\varphi_\varepsilon-\varphi\|_\infty\leqslant  \varepsilon$ and $H\big(x,D_x \varphi_\varepsilon\big)\leqslant  c_0+\varepsilon$\ for all $x\in M$, given by Theorem \ref{approx}. For $\varepsilon$ small enough, we have $\varphi_\varepsilon<0$. By (\ref{e/e}) we get
\begin{align*}
&\frac{1}{\alpha_n}\bigg(u_{\lambda_n}\big(\gamma_n(0)\big)-e^{-\lambda_n \int_{-t_n}^0a(\gamma_n(\tau))d\tau}u_{\lambda_n}\big(\gamma_n(-t_n)\big)\bigg)
=
\int_{TM}\big(L(x,v)+c_0\big)d\tilde{\mu}_n\\
&\geqslant  \int_{TM}\Big( D_x\varphi_\varepsilon(v)-H\big(x,D_x \varphi_\varepsilon\big)+c_0\Big)\,d\tilde{\mu}_n
\geqslant
\int_{TM}\big(  D_x\varphi_\varepsilon(v) -\varepsilon\big)\, d\tilde{\mu}_n\\
&=\frac{\varphi_\varepsilon\big(\gamma_n(0)\big)-e^{-\lambda_n\int_{-t_n}^0a(\gamma_n(\tau))d\tau}\varphi_\varepsilon\big(\gamma_n(-t_n)\big)}{\alpha(\lambda_n,\gamma_n,t_n)}-\lambda_n\int_{TM}a(x)\varphi_\varepsilon(x)\,d\tilde{\mu}_n-\varepsilon\\
&\geqslant  \frac{\varphi_\varepsilon\big(\gamma_n(0)\big)}{\alpha_n}-\lambda_n\int_{TM}a(x)\varphi_\varepsilon(x)\,d\tilde{\mu}_n-\varepsilon,
\end{align*}
where for the first inequality we have used Fenchel's inequality \eqref{fenchel} and for the subsequent equality an integration by parts.
Let $\varepsilon\to 0^+$ to get
\begin{eqnarray*}
\frac{1}{\alpha_n}\bigg(u_{\lambda_n}\big(\gamma_n(0)\big)-e^{-\lambda_n \int_{-t_n}^0a(\gamma_n(\tau))d\tau}u_{\lambda_n}\big(\gamma_n(-t_n)\big)\bigg)
\geqslant
\frac{\varphi\big(\gamma_n(0)\big)}{\alpha_n}-\lambda_n\int_{TM}a(x)\varphi(x)\,d\tilde{\mu}_n.
\end{eqnarray*}
Now we divide by $\lambda_n$, we use the fact that $\lambda_n\alpha_n\to+\infty$ and Lemma \ref{a/b} to get
\[\int_{TM}a(x)\varphi(x)\,d\tilde\mu(x,v)\geqslant  -\|a\|_\infty \|u_\lambda\|_\infty.\]
Since for, all $m\geqslant  0$, $\varphi-m\leqslant  0$ is also a subsolution of (\ref{e0}), we get, by applying the previous inequality to
$\varphi-m$, that
\[
\|\varphi\|_\infty\geqslant  \int_{TM}a(x)\varphi(x)\,d\tilde{\mu}(x,v)
\geqslant  m\int_{TM}a(x)\,d\tilde{\mu}(x,v)-\|a\|_\infty \|u_\lambda\|_\infty,\quad \forall m\geqslant  0,
\]
which implies that $\int_{TM}a(x)\,d\tilde{\mu}\leqslant  0$. This leads to a contradiction with \eqref{V>0} being $\tilde\mu$ a Mather measure, as we recalled above.

Let us fix $\lambda\in (0,\lambda_0)$ and $x\in M$.
By sending $t\to +\infty$ in \eqref{ulam} and by exploiting the information just gathered, we derive that
\begin{eqnarray*}
u_\lambda(x)
=
\lim_{t\to +\infty} \int_{-t}^0 e^{-\lambda \int_s^0a(\gamma^x_\lambda(\tau))d\tau}\Big(L\big(\gamma^x_\lambda(s),\dot\gamma^x_\lambda(s)\big)+c_0\Big)ds.
\end{eqnarray*}
We conclude that
\begin{align}\label{eq minimizing curve}
u_\lambda(x)=\lim_{t\to +\infty} \int_{-t}^0 e^{-\lambda \int_s^0a(\gamma^x_\lambda(\tau))d\tau}\Big(&L\big(\gamma^x_\lambda(s),\dot\gamma^x_\lambda(s)\big)+c_0\Big)ds\\
&=
 \int_{-\infty}^0 e^{-\lambda \int_s^0a(\gamma^x_\lambda(\tau))d\tau}\Big(L\big(\gamma^x_\lambda(s),\dot\gamma^x_\lambda(s)\big)+c_0\Big)ds.\nonumber
\end{align}
Indeed \eqref{eq L^1 estimate} means that the positive function $s\mapsto e^{-\lambda \int_s^0a(\gamma^x_\lambda(\tau))d\tau}$ is in $L^1\big((-\infty,0]\big)$. Furthermore, the fact that the curve $\gamma^x_\lambda$ is Lipschitz implies that the
function $s\mapsto L\big(\gamma^x_\lambda(s),\dot\gamma^x_\lambda(s)\big)$ is in $L^\infty\big((-\infty,0]\big)$.  Equality \eqref{eq minimizing curve} follows from this by making use of the Dominated Convergence Theorem.
\end{proof}

\subsection{Diverging families of solutions}\label{sec model diverging}

We have seen in Theorem \ref{thm3} that, under the assumption (\ref{V>0}), the maximal solution $u_\lambda$ of (\ref{VH}) uniformly converges as $\lambda\to 0^+$ to a critical solution $u_0$. In this section, we will show that if we furthermore assume
\begin{itemize}
\item [($a$2)\,] there is a point $x_0\in M$ such that $a(x_0)<0$;\smallskip
\item [($a$3)\,] $a(x)\geqslant  0$ in an open neighborhood $U$ of the projected Aubry set $\mathcal A$,\smallskip
\end{itemize}
then there exists a family of solutions to (\ref{VH}) that uniformly diverges to $-\infty$. Furthermore, when condition
(a$3$) is reinforced in favor of the following one
\begin{itemize}
\item [($a$3$'$)] $a(x)>0$ on $\mathcal A$,
\end{itemize}
then all solutions to \eqref{VH} different from the maximal ones uniformly diverge to  $-\infty$ as $\lambda\to 0^+$. We summarize all this in the following statement.  We refer the reader to Section
\ref{sec weak kam} for the definition of the projected Aubry set. We underline that, in view of Theorem \ref{thm M in A}, condition (a$3$) is stronger
than the integral condition \eqref{V>0}. Throughout this section, we will denote by $u_\lambda$ the maximal solution of \eqref{VH}.

\begin{theorem}\label{thm dichotomy}
Let us assume ($a$2).
\begin{itemize}
\item[\em (i)] If ($a$3) holds, there is a family of solutions $(v_\lambda)_{\lambda\in (0,\hat \lambda)}$ of (\ref{VH}) for some $\hat\lambda\in (0,1)$ uniformly diverging to $-\infty$ as $\lambda\to 0^+$.\smallskip
\item[\em (ii)] If ($a$3$\,'$) holds, then any family of solutions $(v_\lambda)_{\lambda\in (0,\lambda')}$ of (\ref{VH}) with $\lambda'\in (0,1)$  satisfying $v_\lambda\not=u_\lambda$ for all $\lambda\in (0,\lambda')$ uniformly diverges to $-\infty$ as $\lambda\to 0^+$.
\end{itemize}
\end{theorem}

\begin{proof}
By Theorem \ref{thm outA}, there is a subsolution $\varphi$ of (\ref{e0}) which is $C^\infty$ and strict in $M\setminus \A$. Therefore, for every
open neighborhood $U$ of $\A$ there is a $\delta=\delta(U)>0$ depending on $U$ such that
\[
H\big(x,D_x \varphi\big)\leqslant  c_0-\delta  <c_0 \quad \hbox{for all $x\in M\setminus U$.}
\]

Let us prove {\em (i)}. Let $U$ be the open neighborhood of $\A$ given by condition ($a$3). Define
\[\varphi_\lambda:=\varphi-\frac{1}{\sqrt{\lambda}}.\]
For $\hat \lambda>0$ small enough, we have $\varphi_\lambda<0$ for all $\lambda\in(0,\hat \lambda)$. Up to choosing a smaller $\hat \lambda>0$ if necessary, we have
\[
\lambda a(x)\varphi_\lambda+H(x,D_x \varphi_\lambda )\leqslant  \lambda a(x)\varphi_\lambda(x)+c_0\leqslant  c_0\quad \hbox{a.e. in $U$,}
\]
and
\[
\lambda a(x)\varphi_\lambda+H(x,D_x \varphi_\lambda )\leqslant  \lambda\|a\|_\infty \|\varphi\|_\infty+\|a\|_\infty\sqrt{\lambda}+c_0-\delta  <c_0\quad\hbox{in $M\backslash U$,}
\]
for all $\lambda\in(0,\hat \lambda)$. Therefore, $\varphi_\lambda$ is a subsolution of (\ref{VH}). Now we denote by $T^{\lambda,-}_t$ (resp. $T^{\lambda,+}_t$) the Lax-Oleinik semigroup defined in (\ref{T-}) (resp. the forward semigroup defined in (\ref{T+})) associated to $L(x,v)-\lambda a(x)u$. Define
\[v^+_\lambda:=\lim_{t\to+\infty}T^{\lambda,+}_t\varphi_\lambda,\quad v_\lambda:=\lim_{t\to+\infty}T^{\lambda,-}_tv^+_\lambda,\]
By Lemma \ref{T-T+}, $v^+_\lambda\leqslant  \varphi_\lambda$, in particular $v^+_\lambda$ uniformly diverges to $-\infty$ as $\lambda\to 0^+$.
By Lemma \ref{u-v-}, $v_\lambda$ is a solution of (\ref{VH}), and there is a point $x_\lambda\in M$ such that
$v_\lambda(x_\lambda)=v^+_\lambda(x_\lambda)$ for each $\lambda\in (0,\hat\lambda)$. We derive that $v_\lambda(x_\lambda)\to -\infty$ as $\lambda\to 0^+$.
By Proposition \ref{uto-inf}, we conclude that $v_\lambda$ uniformly diverges to $-\infty$ as $\lambda\to 0^+$.\smallskip

Let us prove {\em (ii)}. From assumption  ($a$3$'$)  we infer that there is an open neighborhood $U$ of $\A$ and $\theta>0$ such that
$a(x)\geqslant  \theta  $ for all $x\in U$. One can easily check that there exist $\lambda'\in(0,1)$ small enough and $\varepsilon>0$ such that
\begin{eqnarray*}
\lambda a(x)\varphi_\lambda+H(x,D_x \varphi_\lambda )\leqslant  \lambda a(x)\varphi_\lambda(x)+c_0\leqslant  \lambda\|a\|_\infty\|\varphi\|_\infty-\sqrt{\lambda}\,\theta  +c_0\leqslant  c_0-\varepsilon\quad \hbox{a.e. in $U$}
\end{eqnarray*}
and
\begin{eqnarray*}
\lambda a(x)\varphi_\lambda+H(x,D_x \varphi_\lambda )\leqslant  \lambda\|a\|_\infty \|\varphi\|_\infty+\|a\|_\infty\sqrt{\lambda}+c_0-\delta \leqslant  c_0-\varepsilon \quad\hbox{in  $M\backslash U$,}
\end{eqnarray*}
for all $\lambda\in(0,\lambda')$. Therefore, $\varphi_\lambda$ is a strict subsolution of (\ref{VH}).\medskip\\
\begin{claim}
{\em Let $\lambda\in(0,\lambda')$. There is no solution $v_\lambda\not=u_\lambda$ of (\ref{VH}) satisfying
$\varphi_\lambda\leqslant  v_\lambda \leqslant  u_\lambda$ in $M$.}\medskip
\end{claim}
\noindent We argue by contradiction. Assume there is such a solution $v_\lambda $. We have
\[T^{\lambda,-}_t\varphi_\lambda\leqslant  T^{\lambda,-}_t v_\lambda \leqslant  T^{\lambda,-}_t u_\lambda=u_\lambda \qquad \forall t>0.\]
Since $\varphi_\lambda$ is a strict subsolution of (\ref{VH}), by Lemma \ref{u-v-},
\[\lim_{t\to+\infty}T^{\lambda,-}_t\varphi_\lambda=u_\lambda.\]
We get
\[\lim_{t\to+\infty}T^{\lambda,-}_t v_\lambda =u_\lambda,\]
which contradicts the fact that $v_\lambda $ is a fixed point of $T^{\lambda,-}_t$.\smallskip

Therefore, for each solution $v_\lambda $ of (\ref{VH}), there is a point $x_\lambda\in M$ such that
\[
v_\lambda (x_\lambda)\leqslant  \varphi_\lambda(x_\lambda)\to -\infty\quad\hbox{as $\lambda\to 0^+$}.
\]
By Proposition \ref{uto-inf}, $v_\lambda $ uniformly converges to $-\infty$ as $\lambda\to 0^+$.
\end{proof}

\begin{remark}
We note that the proof above also shows that the solutions $v_\lambda$  diverge to $-\infty$ with speed (at least) of order $-1/\sqrt{\lambda}$.
\end{remark}

\section*{Acknowledgements}
Andrea Davini is a member of the INdAM Research Group GNAMPA. He is supported for this research by Sapienza Universit\`a di Roma - Research Funds 2021. Panrui Ni and Maxime Zavidovique are supported by ANR CoSyDy (ANR-CE40-0014). Jun Yan is supported by National Natural Science Foundation of China (Grant Nos. 12171096, 12231010).

\section*{Declarations}

\noindent {\bf Conflict of interest statement:} The author states that there is no conflict of interest.

\medskip

\noindent {\bf Data availability statement:} Data sharing not applicable to this article as no datasets were generated or analysed during the current study.

\bibliography{discount}
\bibliographystyle{siam}

\end{document}